\documentclass[reqno]{amsart}
\usepackage{amsmath,amssymb,cite}
\usepackage[mathscr]{euscript}

\theoremstyle{plain}
\newtheorem{theorem}{Theorem}[section]
\newtheorem{proposition}[theorem]{Proposition}
\newtheorem{lemma}[theorem]{Lemma}

\theoremstyle{definition}

\newtheorem{assumption}[theorem]{Assumption}

\theoremstyle{remark}

\numberwithin{equation}{section}
\numberwithin{theorem}{section}

\def\be{\begin{equation}}
\def\ee{\end{equation}}
\def\bp{\begin{pmatrix}}
\def\ep{\end{pmatrix}}
\def\bea{\begin{eqnarray}}
\def\eea{\end{eqnarray}}

\def\\{\par\medskip}

\let\0=\noindent

\newcommand{\mc}[1]{{\mathcal #1}}

\newcommand{\bs}[1]{{\boldsymbol #1}}
\newcommand{\bb}[1]{{\mathbb #1}}

\newcommand{\rmi}{\mathrm{i}}
\newcommand{\rmd}{\mathrm{d}}

\let\eps=\varepsilon

\newcommand{\id}{{1 \mskip -5mu {\rm I}}}
\renewcommand{\eps}{\varepsilon}
\renewcommand{\hat}{\widehat}

\title{Stochastic Allen-Cahn equation with mobility}

\author[L.\ Bertini]{Lorenzo Bertini}
\address{Lorenzo Bertini \hfill\break \indent
   Dipartimento di Matematica, 
   Universit\`a di Roma `La Sapienza' 
   \hfill\break \indent
   P.le Aldo Moro 5, 00185 Roma, Italy}
 \email{bertini@mat.uniroma1.it}
\author[P.\ Butt\`a]{Paolo Butt\`a}
\address{Paolo Butt\`a\hfill\break \indent
   Dipartimento di Matematica, 
   Universit\`a di Roma `La Sapienza' 
   \hfill\break \indent
   P.le Aldo Moro 5, 00185 Roma, Italy}
 \email{butta@mat.uniroma1.it}
\author[A.\ Pisante]{Adriano Pisante}
\address{Adriano Pisante \hfill\break \indent
   Dipartimento di Matematica, 
   Universit\`a di Roma `La Sapienza' 
   \hfill\break \indent
   P.le Aldo Moro 5, 00185 Roma, Italy}
\email{pisante@mat.uniroma1.it}

\begin{document}

\begin{abstract}
  We introduce a class of stochastic Allen-Cahn equations with a
  mobility coefficient and colored noise. For initial data with finite
  free energy, we analyze the corresponding Cauchy problem on the
  $d$-dimensional torus in the time interval $[0,T]$. Assuming that
  $d\le 3$ and that the potential has quartic growth, we prove
  existence and uniqueness of the solution as a process $u$ in $L^2$
  with continuous paths, satisfying almost surely the regularity
  properties $u\in C([0,T]; H^1)$ and $u\in L^2([0,T];H^2)$.
\end{abstract}

\keywords{Stochastic PDEs, Allen-Cahn equation, Well-posedness}
\thanks{The present work was financially supported  by  PRIN 20155PAWZB ``Large Scale Random Structures''. }

\maketitle
\thispagestyle{empty}
 
\section{Introduction}
\label{sec:1}

The analysis of stochastic perturbations of the Allen-Cahn equation,
due to their relevance both from a theoretical and applied viewpoint,
has been a main topic in the development of the theory of stochastic
partial differential equations. We consider the case in which the
space variable belongs to the $d$-dimensional torus $\bb T^d :=\bb
R^d/ \bb Z^d$. The typical setting is the following. Fix a smooth
double well potential $W\colon \bb R\to \bb R$ and a filtered
probability space equipped with a cylindrical Wiener process
$\alpha$. A class of stochastic perturbations of the Allen-Cahn
equation is then given by
\begin{equation}
\label{i.1}
\rmd u_t = \big( \Delta u_t - W' (u_t) \big) \rmd t + \sqrt{2} j*\rmd\alpha_t \;.  
\end{equation}
Here, the unknown $u=u_t(x)$, $(t,x)\in[0,T]\times \bb T^d$, $T>0$, is
real-valued and it represents the local order parameter, $\Delta$ is
the Laplacian, $j=j(x) \colon \bb T^d \to \bb R$, and $*$ denotes
convolution in the space variable. The case of perturbation by
space-time white noise is formally recovered when $j$ is the Dirac's
delta function.

The so-called semigroup approach \cite{DaZ} to the analysis of the
stochastic Allen-Cahn consists in writing the Cauchy problem with
initial datum $\bar u_0$ associated to \eqref{i.1} in the mild form,
i.e.,
\begin{equation}
  \label{i:mf}
  u_t = e^{t\Delta} \bar u_0 - \int_0^t\! e^{(t-s)\Delta} W'(u_s) \, \rmd s 
  + \sqrt{2}\int_0^t\! e^{(t-s)\Delta} j* \rmd\alpha_s\;,
\end{equation}
where $e^{t\Delta}$ denotes the heat semigroup. In the one-dimensional
case with $j=\delta$ or in $d>1$ with $j$ smooth enough, the last term
on the right-hand side of \eqref{i:mf} is, with probability one, a
process in $C(\bb T^d)$ with continuos paths. By a fixed point
argument in $C([0,T];C(\bb T^d))$, it is then possible to prove
existence and uniqueness to \eqref{i:mf}, for almost all realizations
of the noise, see, e.g., \cite{Cerrai}, where a more general setting is
considered. When $W'$ is Lipschitz, this approach applies also to the
case in which the state space is $L^2(\bb T^d)$ instead of $C(\bb
T^d)$, and the same holds even when $W'$ has polynomial growth, relying on the one-side Liptschitz property of $W'$ \cite{DaZ,CD}.

Considering still the case with $W'$ Lipschitz and with the same
restrictions on $j$, the stochastic Allen-Cahn equation \eqref{i.1}
can be also analyzed using the so-called variational approach
\cite{KR,PrRo}. This approach relies on the embeddings $H^1(\bb T^d)
\subset L^2(\bb T^d) \subset H^{-1}(\bb T^d)$, the Cauchy problem
associated to \eqref{i.1} is then understood as the following equality
in $H^{-1}(\bb T^d)$,
 \begin{equation}
 \label{i:vf}
 u_t= \bar{u}_0 + \int_0^t\! \big[\Delta u_s - W'(u_s)\big] \, \rmd s +  \sqrt{2} j * \alpha_t \;.
 \end{equation}
The main step for existence is an It\^o's formula for the map $u\mapsto \|u\|_{L^2(\bb T^d)}^2$, which yields the a priori bounds needed to construct the solution $u$, by compactness arguments, as a process in $L^2(\bb T^d)$ with continuous paths and $u\in L^2([0,T];H^1(\bb T^d))$ with probability one. More recently, in \cite{LR} the variational approach has been extended to the case of $W'$ with some polynomial growth again in view of the one-side Liptschitz property.   

Approximation of the Allen-Cahn equation \eqref{i.1} by time discretization has been considered in \cite{KLL}, in terms of the backward Euler scheme; indeed, as the time step goes to zero, this method recovers the unique solution discussed above. Similar time and space discretization of \eqref{i.1} were previously investigated first in, e.g., \cite{H} under Lipschitz assumption on the nonlinear term and extended in \cite{J} when $W'$ has polynomial growth.   

In the case of perturbation by space-time white noise, $j=\delta$, and $d>1$, the last term in the right-hand side of \eqref{i:mf} is, with probability one, only a distribution and the well-posedness of the stochastic perturbation of the Allen-Cahn equation becomes a major issue.  In particular, to make sense of the equation a proper renormalization of the non linear term $W'$ is needed.  In dimension $d=2$, when $W$ is a polynomial, this renormalization amounts to the Wick ordering \cite{AR,DD,JM}. In dimension $d=3$, the renormalization of the non linearity is more involved; for a quartic potential $W$, a local existence and uniqueness result is proven in \cite{Hairer}, and it has been extended in \cite{MW3} to arbitrary time intervals.

Regarding the choice of the random forcing term in \eqref{i.1}, we would like to make the following model remark. The choice of the space-time white noise has the doubtless appeal of simplicity and universality, and it is really mandatory when \eqref{i.1} is used in the framework of stochastic quantization or to model dynamical critical fluctuations \cite{HH}. In the latter case, the potential $W$ is not arbitrary but the quartic potential. Indeed, as shown in \cite{BPRS,FR} for $d=1$ and in \cite{MW2} for $d=2$, with these choices \eqref{i.1} describes the asymptotic of the fluctuations at the critical point for a Glauber dynamics with local mean field interaction. On the other hand, if we regard \eqref{i.1} as a phenomenological model for phase segregation and interface dynamics, the choice of a noise with nonzero spatial correlation length, i.e., a smooth $j$, is not unsound since we are going to look at the order parameter on larger space scales. Analogously, any reasonable double well potential $W$ will yield essentially the same limiting behavior.
 
The deterministic Allen-Cahn equation, i.e., \eqref{i.1} with $j=0$, can be viewed as the $L^2$-gradient flow of the van der Waals free energy functional,
\begin{equation}
\label{IntF}
\mc F (u) := \int\!\Big[ \frac12 |\nabla u|^2 + W(u) \Big]\,\rmd x\;.
\end{equation}
Correspondingly, in the case when $j$ is the Dirac's delta function, the process $u$ is (informally) reversible with respect to the (informal) probability measure $P(\mc D u)\propto \exp\big\{ -\mc F(u)\big\}\mc D u$.

With respect to the setting described above, in this paper we analyze a stochastic Allen-Cahn equation in which we introduce a mobility coefficient, that is,
\begin{equation}
\label{i.2}
\rmd u_t = \sigma(u_t) \big(  \Delta u_t - W' (u_t) \big) \rmd t  +
\sqrt{2 \sigma(u_t)}\, j*\rmd \alpha_t\;, 
\end{equation}
where the \emph{mobility} $\sigma\colon \bb R\to \bb R_+$ is smooth, bounded, and uniformly strictly positive. Moreover, $W$ is convex at infinity with at most quartic growth. In terms of gradient flows, \eqref{i.2} with $j=0$ is the gradient flow of $\mc F$ in $L^2(\sigma(u)^{-1} \rmd x)$. Finally, the choice of the random forcing term in \eqref{i.2} is suggested by the case of constant mobility. Indeed, when $\sigma$ is constant and $j$ is the Dirac's delta function, the process $u$ is still (informally) reversible with respect to the (informal) probability $P(\mc D u)\propto \exp\big\{ -\mc F(u)\big\}\mc D u$ regardless of the specific value of $\sigma$. In the physical literature, see e.g., \cite[Sect.\ IV.A.1]{HH} or \cite[Sect.\ II.7.3]{Spohn}, this choice is usually referred to as the Onsager's prescription.

A motivation for the introduction of the mobility in the Allen-Cahn
equation relies in the analysis of the corresponding sharp interface
limits. For instance, as well known, for suitably prepared initial
data, in this singular limit the deterministic Allen-Cahn equation
(with constant mobility) converges to the motion by mean curvature,
see e.g. \cite{ESS, Ilmanen}. As discussed in \cite[\S~4]{S}, this
approximation to motion by mean curvature has the peculiar feature of
exhibiting a trivial transport coefficient in the limiting
evolution. On the other hand, when a non-constant mobility coefficient
is introduced as in \eqref{i.2}, we expect that the limiting interface
evolution is described by motion by mean curvature with a non-trivial
transport coefficient satisfying the corresponding Einstein's relation
\cite[\S~3]{S} (see \cite{DOPT, KS} for the case of a non-local
equation). As far as the stochastic Allen-Cahn equation is considered,
a relevant issue is the large deviation asymptotics in such sharp
interface limit. In the case of constant mobility, this analysis is
carried out in \cite{BBP2}, see also the related discussion in \cite{KORV}. 

To our knowledge, the stochastic Allen-Cahn equation with mobility has
not been discussed in the literature. 
In this paper, we consider the Cauchy problem associated to \eqref{i.2} with
initial datum $\bar u_0\in H^1(\bb T^d)$ when $d\le 3$, the potential 
$W$ is convex at infinity with at most quartic growth, and $j$
belongs to the Sobolev space $H^1(\bb T^d)$.
We prove the existence and uniqueness of the solution as a process $u$
in $L^2(\bb T^d)$ with continuous paths satisfying $u\in C([0,T];
H^1(\bb T^d))\cap L^2([0,T];H^2(\bb T^d))$ almost surely and such that the corresponding norms are random variables whose moments are all finite.

The semigroup approach does not seem to be applicable to equation \eqref{i.2}, first because it cannot be recasted in a mild form in terms of a linear semigroup (the diffusion term is now nonlinear), but also because the reaction term $-\sigma W'$ no longer satisfies the one-side Lipschitz property. On the other hand, our result seems difficult to obtain by the variational approach discussed above even in the case of constant mobility,  
 see the discussion at the end of Section \ref{sec:2}. 

The restriction $d\le 3$ is connected to the quartic growth of the
potential, allowing to control some non-linear terms via Sobolev
embeddings. The choice of periodic boundary conditions does simplify
computations, but the arguments here presented are robust enough to be
adapted to the case of a bounded domain with either Dirichlet or
Neumann boundary conditions.

From a technical viewpoint, existence of solutions to \eqref{i.2} will be proven by a compactness argument on suitable approximate solutions in the same spirit of the variational approach. More precisely, the approximate solutions are constructed by time discretization of the mobility coefficient and regularizing the nonlinear term. The necessary a-priori bounds are obtained, taking full advantage of the variational structure of the equation, by deriving an It\^o's formula for suitable approximations of the map $u\mapsto \mc F(u)$ defined in \eqref{IntF}. Uniqueness will be achieved by an $H^{-1}$ estimate inspired by the one in \cite{AL} for similar deterministic evolution equations, together with a Yamada-Watanabe type argument. 
 
\section{Notation and results}
\label{sec:2}

Throughout this paper we shall shorthand $L^p=L^p(\bb T^d)$, $p\in [1,+\infty]$, and let $H^s=H^s(\bb T^d)$, $s\in \bb R$, be the fractional Sobolev space. Moreover, given $T>0$ we also shorthand $C(L^p)=C([0,T];L^p)$, $C(H^s)= C([0,T];H^s)$, and $L^p(H^s)= L^p([0,T];H^s)$. 

We consider the following stochastic partial differential equation, 
\begin{equation}
\label{1}
\rmd u = \sigma(u) \left[\Delta u -  W'(u) \right]\rmd t  + \rmd M\;, 
\end{equation}
where, for $\varphi\in L^2 $, $M^\varphi_t:=\langle M_t,\varphi\rangle_{L^2}$, $t\ge 0$, is a continuous square integrable martingale with quadratic variation,
\begin{equation}
\label{2}
\big[M^\varphi\big]_t = 2 t \int\! \big[j *(\sqrt{\sigma(u)}\,\varphi)\big]^2 \, \rmd x \;. 
\end{equation}
Here $j \in H^1$ is a fixed function, $*$ denotes the convolution on $\bb T^d$, and the following conditions on the potential $W$ and the mobility $\sigma$ are assumed to hold.

\begin{assumption}[Assumptions on $W$ and $\sigma$]$~$
\label{t:ws}
\begin{enumerate} 
\item 
  $W\in C^2\big(\bb R ;[0,+\infty)\big)$ and $W$ is uniformly convex at infinity, i.e.,
  there exists a constant $C\in (0,+\infty)$ and a compact $K\subset
  \bb R$ such that $W''(u) \ge \frac{1}{C}$ for any $u\not\in K$.
\item $W$ has at most growth $4$, i.e., there
  exists a constant $C\in (0,+\infty)$ such that $|W(u)|\le C
  (|u|^4+1)$ for any $u\in \bb R$. 
\item $W'$ has at most growth $3$, i.e., there
  exists a constant $C\in (0,+\infty)$ such that $|W'(u)|\le C
  (|u|^{3}+1)$ for any $u\in \bb R$. 
\item There
  exists a constant $C\in (0,+\infty)$ such that 
  $|W''(u)|\le C (\sqrt{W(u)}+1)$ for any $u\in \bb R$. 
\item $\sigma\in C^2(\bb R)$, $\sigma$ is bounded and
  uniformly strictly positive, i.e., there exists a constant $C\in
  (0,+\infty)$ such that $\frac 1C \le \sigma(u)\le C$ for any $u\in
  \bb R$.
\item $\sigma'$, $\sigma''$ are bounded.
\end{enumerate}
\end{assumption}

We prove the existence and uniqueness of the Cauchy problem associated to \eqref{1} with a deterministic initial datum $\bar u_0 \in H^1$ in space dimensions $d\le 3$. To formulate the precise result we introduce two different notions of solution.

Given $T>0$, we consider $C(L^2) \equiv C([0,T];L^2)$, endowed with the norm topology, the associated Borel $\sigma$-algebra $\mc B$, and the canonical filtration $\mc B_t$, $t\in [0,T]$. The canonical coordinate on $C(L^2)$ is denote by $u=(u_t)_{t\in[0,T]}$.

Given $\bar u_0\in H^1$, a probability $\bb P$ on $C(L^2)$ solves the \emph{martingale problem} associated to \eqref{1} with initial datum $\bar u_0$ iff $\bb P (u_0=\bar u_0)=1$, $\bb P (u\in L^\infty(H^1) \cap L^2(H^2))=1$, and for each $\psi\in C^\infty([0,T]\times \bb T^d)$ the process,
\begin{equation}
\label{a0}
M^\psi_t := \int\! u_t \psi_t  \, \rmd x -\int\! u_0 \psi_0 \, \rmd x - \int_0^t\! \int\! \Big[u_s \partial_s\psi_s + \sigma(u_s)\Big(\Delta u_s - W'(u_s)\Big) \psi_s \Big] \, \rmd x  \, \rmd s
\end{equation}
is a continuous, square integrable $\bb P$-martingale with quadratic variation,
\begin{equation}
\label{a0.5}
\big[M^\psi\big]_t = 2  \int_0^t \int\! \big[ j *\big( \sqrt{\sigma(u_s)} \psi_s\big)\big]^2 \, \rmd x  \, \rmd s\;. \end{equation}

We shall refer to such probability $\bb P$ as a \emph{martingale solution} to \eqref{1} with initial datum $\bar u_0$. \emph{Uniqueness in law} (or uniqueness of martingale solutions) holds whenever there exists at most one probability on $C(L^2)$ meeting the above requirements.

To introduce the notion of strong solution it is first necessary to construct the martingale in terms of cylindrical Wiener process, whose definition we next recall. A $L^2$-cylindrical Wiener process on the probability space $(\Omega,\mc G,\mc P)$ is a measurable map $\alpha\colon \Omega\to C(H^{-\bar{s}})$, $\bar{s}>d/2$, such that $\alpha_t$, $t\in [0,T]$, is a mean zero Gaussian process with covariance,
\begin{equation*}
\mc E \big( \alpha_t(\phi) \alpha_{t'}(\phi') \big) = t\wedge t'\, \langle\phi , \phi'\rangle_{L^2} = t\wedge t'\,\langle\phi, (\textrm{Id}-\Delta)^{-\bar s} \phi'\rangle_{H^{\bar{s}}}\;, \qquad \phi,\phi'\in H^{\bar{s}}\;,
\end{equation*}
where $\mc E$ denotes the expectation with respect to $\mc P$. A $L^2$-cylindrical Wiener process can be constructed as $\alpha_t = \sum_k \beta^k_t e_k$, where $\{e_k\}$ is an orthonormal basis in $L^2$ and $\{\beta^k\}$ are independent standard Brownian processes on $\bb R$. Note that, since the embedding $L^2\hookrightarrow H^{-s}$ is Hilbert-Schmidt for $s>d/2$, $(\textrm{Id}-\Delta)^{-\bar s}$ is trace-class on $H^{-\bar{s}}$. We refer to \cite{DaZ} for a general overview on infinite dimensional stochastic calculus. We denote by $\{\mc G^\alpha_t\}$ the filtration generated by $\alpha$ completed with respect to $\mc P$. 

Given $v\in L^2$, we let $B(v) \colon L^2 \to L^2$ be the linear operator defined by $B(v)\psi = \sqrt{2 \sigma(v)}\, j *  \psi$. Since $j\in H^1$, $B(v)$ is Hilbert-Schmidt, i.e., $\mathrm{Tr}_{L^2}(B(v)B(v)^*) < \infty$.   
 
Given $\bar u_0\in H^1$, a measurable map $u\colon \Omega \to C(L^2)$ is a \emph{strong solution} to \eqref{1} with initial datum $\bar u_0$ iff $u$ is a $\mc G^\alpha_t$-adapted process, $\mc P(u_0=\bar u_0)=1$, $\mc P(u \in L^\infty(H^1) \cap L^2(H^2)) =1$, and, for each $\psi\in C^\infty\big([0,T]\times \bb T^d\big)$ and $t\in [0,T]$, the following equality holds $\mc P$-a.s.,
\be
\label{form}
\begin{split}
\langle u_t, \psi_t \rangle_{L^2} & = \langle u_0, \psi_0\rangle_{L^2}  + \int_0^t\! \langle u_s, \partial_s\psi_s\rangle_{L^2} \,\rmd s \\ & \quad + \int_0^t\! \langle \sigma(u_s) (\Delta u_s - W'(u_s)), \psi_s\rangle_{L^2} \, \rmd s + \int_0^t\! \langle \psi_s, B(u_s)\,\rmd\alpha_s \rangle_{L^2}\;,
\end{split}
\ee
where the last term is understood as an It\^o stochastic integral, see \cite{DaZ}. Note that \eqref{form} corresponds to \eqref{i.2} tested with the function $\psi$. \emph{Uniqueness of strong solutions} holds if any two such solutions $u,u'$ satisfy $\mc P(u_t=u'_t \; \;\forall\, t\in [0,T])=1$. 

It is worthwhile to observe that the requirement that strong solutions are adapted to the filtration generated by $\alpha$ means that they can be obtained as non-anticipati\-ve functions of $\alpha$.

In the analysis of \eqref{1} two specific functionals play an essential role, the aforementioned van der Waals' free energy functional $\mc F \colon L^2 \to [0,+\infty]$, defined by 
\be
\label{F}
\mc F(u) := \begin{cases} {\displaystyle \int\!\Big[ \frac12 |\nabla u|^2 + W(u) \Big]\,\rmd x} & \text{ if $u\in H^1$}, \\ +\infty & \text{ otherwise}, \end{cases}
\ee
and the Wilmore functional $\mc W \colon L^2 \to [0,+\infty]$, defined by 
\be
\label{W}
\mc W(u) := \begin{cases} {\displaystyle \int\!\sigma(u) \Big[\Delta u - W'(u) \Big]^2\,\rmd x} & \text{ if $u\in H^2$}. \\ +\infty & \text{ otherwise}. \end{cases}
\ee
Observe that since $W$ has at most quartic growth and $d\le 3$, by Sobolev embedding, $u\in H^1$ implies $W(u)\in L^1$, and, more precisely, $\mc F(u) \leq C(1+\|u\|_{H^1}^4) $. Similarly, since $W'$ has at most cubic growth, again by Sobolev embedding, if $u \in H^2$ then $W'(u)\in L^2$ and $\mc W(u) \leq C(1+ \| u\|_{H^2}^2+\| u\|_{H^1}^6)$. 

\begin{theorem}
\label{t:eu}
Given $\bar u_0\in H^1$, there exists a unique martingale solution $\bb P$ to \eqref{1}. Moreover, $\bb P (u \in C(H^1))=1$ and for  $p\in [1,\infty)$ there exists $C=C(\bar{u}_0,T,p)>0$ such that 
\begin{equation}
\label{steu}
\bb E \Big( \sup_{t\in [0,T]} \mc F (u_t) + \int_0^T\! \mc W (u_t)\, \rmd t \Big)^p \le C \;.
\end{equation}

In addition, given a probability space $(\Omega,\mc G,\mc P)$ equipped with a $L^2$-cylindrical Wiener process $\alpha$, there exists a unique strong solution $u$ to \eqref{1} with initial condition $\bar u_0$. The law of $u$ is the martingale solution $\bb P$.
\end{theorem}

In view of the bounds on $\mc F$ and $\mc W$ discussed above, we notice that if $\bb P$ is a martingale solution then the integrand in \eqref{steu} is $\bb P$-a.s.\ finite; the estimate \eqref{steu} states that its moments are finite. We remark that this bound relies on the assumption that the function $j$ in \eqref{2} belongs to $H^1$. In particular, in one dimension, the case of space-time white noise is not covered by the previous theorem. In the case of constant mobility the corresponding solution $u$ exists, e.g., in $C([0,T];C( \bb T))$, but the $H^1$ norm of $u_t$ is infinite almost surely for each $t \in [0,T]$. 

The proof of Theorem \ref{t:eu} is structured through the following steps, in the same spirit of \cite{Mariani}. The existence of a martingale solution is obtained in Section \ref{sec:4} by means of compactness estimates on the laws of a sequence of adapted processes in $C(L^2)$. In order to handle the mobility, these processes are constructed by introducing a time discretization and solving in each time interval  a suitable semilinear approximated versions of \eqref{1} obtained by freezing the mobility and regularizing the reaction term. The actual construction of these approximated processes  requires an existence result for semilinear equations in $C(H^1)$, which is the content of Section \ref{sec:3}. To prove compactness, the key ingredient is the apriori estimate in Lemma \ref{lem:4.1} which basically states that the bound \eqref{steu} holds uniformly in the approximation parameters. This estimate relies on the variational structure of \eqref{1} and its proof is achieved by applying It\^o's formula to a suitable approximation of $\mc F$. To prove uniqueness of martingale solutions, in Section \ref{sec:5} we introduce the notion of weak solution to \eqref{1}; by a martingale representation lemma we show that martingale solutions produce weak solutions. Then, after proving the regularity properties of solutions, we show pathwise uniqueness of weak solutions via $H^{-1}$ estimates. Finally, by adapting the argument in the Yamada-Watanabe theorem, we obtain the existence and uniqueness result as stated in Theorem \ref{t:eu}. Some generation results on a class of $C_0$-semigroups, needed to the theory developed in Section \ref{sec:3}, are stated and proved in Appendix \ref{app:a} by applying the Lumer-Phillips theorem.

As stated before, the Allen-Cahn equation with mobility does not appear to have been considered in the mathematical literature and it does not seem directly analyzable by the variational method. For instance, the one-side Lipschitz condition (see, e.g., \cite[Condition (H3)]{L}) fails for the Gelfand triple $H^1(\bb T^d)  \subset L^2(\bb T^d)\subset H^{-1}(\bb T^d)$ even in the case of $W$ with quadratic growth. On the other hand, in the case of constant mobility the   
method applies, as shown in \cite{LR} and in the subsequent paper \cite{L}, with a dimensional dependent growth condition on the reaction term. In the one dimensional case the cubic growth is covered, however, Theorem \ref{t:eu} provides better regularity properties of the solution $u$, in particular, $u \in L^2([0,T];H^2)$ almost surely. In principle, the latter regularity property could be deduced by working with the Gelfand triple $H^2(\bb T^d) \subset H^1(\bb T^d) \subset L^2(\bb T^d)$, but this would require strong restrictions on the nonlinearity.

 Still in the case of constant mobility, an abstract existence result for stochastic partial differential equation of gradient type is proven in \cite{G}, using an approximation argument relying on apriori bounds analogous to the ones in the present paper.  When applied to the Allen-Cahn equation, see \cite[Rem.~4.9]{G}, the regularity properties are slightly weaker than the ones in Theorem \ref{t:eu}. It would be interesting to generalize the approach of \cite{G} to cover the case of nonconstant mobility.

\section{An auxiliary semilinear equation}
\label{sec:3}

In this section, we provide an existence result for a semilinear equation that will be used to construct an approximation of the stochastic Allen-Cahn equation. The arguments below follow the semigroup approach in \cite{DaZ}, it is however possible to obtain the same result by the variational approach in \cite{PrRo} choosing the Gelfand triple $H^2 \subset H^1\subset L^2$. 

Recall that $(\Omega,\mc G,\mc G_t,\mc P)$ is a standard filtered probability space equipped with a cylindrical Wiener process $\alpha\colon \Omega\to C(H^{-\bar{s}})$, $\bar{s}>d/2$, such that $\mc G_t = \mc G^\alpha_t$ is the filtration generated by $\alpha$ completed with respect to $\mc P$. Let $f\colon \bb R \to \bb R$ be globally Lipschitz. Fix a subinterval $[t_0,t_1] \subset [0, T]$, a $\mc G_{t_0}$-measurable random variable $w\colon \Omega \to H^1$ and a $\mc G_{t_0}$-measurable random variable $v \colon \Omega \to H^2$. Let $\eta>0$ and $R_\eta=(\mathrm{Id}-\eta \Delta)^{-1}$. Consider the following Cauchy problem on the time interval $[t_0,t_1]$,
\begin{equation}
\label{daprato1}
  \begin{cases}
    \rmd u_t = \sigma(v) \big[ \Delta u_t +R_\eta f(R_\eta u_t) \big] \rmd t + \sqrt{ 2 \sigma(v)} j* \rmd\alpha_t \;,\\
    u_{t_0} = w\;.
  \end{cases}
\end{equation}
We say that $u$ is a \emph{strong} solution to \eqref{daprato1} if $u \colon \Omega \to C([t_0, t_1] ; H^1)$ is $\mc G_t$-adapted, $\mc P(u_{t_0}=w)=1$, $\mc P(u \in L^2([t_0,t_1];H^2))=1$, and, for each $\psi\in C^\infty\big([t_0,t_1]\times \bb T^d\big)$ and $t\in [0,T]$, the following equality holds $\mc P$-a.s., 
\begin{equation}
\label{daprato2}
\begin{split}
\langle u_t, \psi_t \rangle_{L^2} & = \langle w, \psi_{t_0} \rangle_{L^2}  + \int_{t_0}^t\! \langle u_s, \partial_s\psi_s\rangle_{L^2} \,\rmd s \\ & \quad + \int_{t_0}^t\! \langle \sigma(v) (\Delta u_s + R_\eta f( R_\eta u_s)), \psi_s\rangle_{L^2} \, \rmd s + \int_{t_0}^t\! \langle \psi_s, B(v)\,\rmd\alpha_s \rangle_{L^2}\;,
\end{split}
\end{equation}
where we recall $B(v) \colon L^2 \to L^2$ is defined by $B(v)\psi = \sqrt{2 \sigma(v)}\, j *  \psi$.
\begin{proposition}
\label{daprato}
Assume the initial datum $w$ in \eqref{daprato1} satisfies $\mc E (\| w\|^2_{H^1})<\infty$. Then, the Cauchy problem \eqref{daprato1} has a strong solution $u$. Moreover, there exists $C>0$ depending only on $\eta$, $\mathrm{Lip}(f)$, and $\mc E (\| w\|^2_{H^1})$ such that
\begin{equation}
\label{daprato3}
\mc E(\| u\|^2_{C([t_0,t_1];H^1)} +\| u\|^2_{L^2([t_0,t_1];H^2)}) \le C \;.
\end{equation} 
Furthermore, if $\mc E (\| w\|^{2p}_{H^1})<\infty$ for some $p>1$ then there exists $C>0$ depending only on $\eta$, $\mathrm{Lip}(f)$, $\mc E (\| w\|^{2p}_{H^1})$, and $p$ such that
\begin{equation}
\label{daprato3p}
\mc E(\| u\|^{2p}_{C([t_0,t_1];H^1)} ) \le C \;.
\end{equation} 
\end{proposition}

\begin{proof}
By Lemma \ref{semigroups}, the operator $A=\sigma(v) \Delta$ generates a $C_0$-semigroup $S(t)$ on the Hilbert space $H^1$. As in Lemma \ref{semigroups}, $\sigma(v) \in H^2$ and it is a bounded multiplier on $H^1$. Moreover, the function $f$ induces a Lipschitz map $F_\eta\colon H^1 \to H^1$ given by $F_\eta(u):=\sigma(v) R_\eta f(R_\eta u)$, in view of the simple estimate,
\[
\begin{split}
\| F_\eta (u_1) -F_\eta (u_2) \|_{H^1} & \le C \| \sigma(v)\|_{H^2} \| R_\eta\|_{_{L^2 \to H^1}} \| f (R_\eta u_1) - f(R_\eta u_2) \| _{L^2} \\ & \le C \| \sigma(v)\|_{H^2} \| R_\eta\|_{_{L^2 \to H^1}} \mathrm{Lip}(f) \| u_1-u_2\|_{L^2}\;.
\end{split}
\]
Finally, by a direct computation, the operator $\tilde{B}\colon L^2 \to H^1$ defined by $\tilde{B} \psi =\sqrt{2 \sigma(v)} j *  \psi$, $\psi \in  L^2$, is an Hilbert-Schmidt operator with norm bound $\|\tilde{B}\|^2_{HS} \le C \| \sqrt{\sigma(v)} \|^2_{H^2} \| j\|_{H^1}^2$.

In view of the previous statements, we can apply \cite[Thm.\ 7.4]{DaZ} and deduce the existence of a unique $\mc G_t$-progressively measurable map $u \colon \Omega \to C([t_0,t_1]; H^1)$ satisfying the mild formulation of \eqref{daprato1}, i.e.,
\begin{equation}
\label{daprato4}
u_t=S(t-t_0)w+\int_{t_0}^t\! S(t-s) F_\eta(u_s) \, \rmd s + \int_{t_0}^t S(t-s) \tilde{B}\, \rmd\alpha_s \, ,
\end{equation}
and the estimate $\sup_{t \in [t_0,t_1]} \mc E (\| u_t\|_{H^1}^2) \le C (1+\mc E (\| w\|_{H^1}^2) )$.

In order to obtain the bound \eqref{daprato3} we would like to apply It\^o's formula to $\Phi(u) := \frac 12 \int\!|\nabla u|^2\,\rmd x$. However, since $\Phi$ is not $C^2$ on $L^2$, to accomplish this step we first introduce a suitable approximation scheme. For $\delta>0$ let $A_\delta$ be as in Lemma \ref{semigroups}, and consider the following linear Cauchy problems,
\begin{equation}
\label{daprato5}
\begin{cases}
\rmd z_t =  \big[A_\delta z_t +F_\eta (u_t) \big]\rmd t + \tilde{B} \rmd\alpha_t\;, \\ z_{t_0} = w\;,
\end{cases}
\end{equation}
where $u$ is the unique solution to \eqref{daprato4}. By Lemma \ref{semigroups}, $A_\delta$ generates a $C_0$-semigroup $S_\delta$ and there exists constants $m_0$ and $C$ such that,
\[
\int_{t_0}^{t_1}\! \| S_\delta (r-t_0) \tilde{B}\|^2_{HS} \, \rmd r \le C (t_1-t_0) e^{m_0(t_1-t_0)} \| \sqrt{\sigma(v)} \|^2_{H^2} \| j\|_{H^1}^2\;.
\]
Therefore, the process $u^\delta_t$  defined by
\begin{equation}
\label{daprato6}
u^\delta_t := S_\delta(t-t_0)w + \int_{t_0}^t\! S_\delta(t-s) F_\eta(u_s) \, \rmd s + \int_{t_0}^t\! S_\delta(t-s) \tilde{B}\, \rmd\alpha_s
\end{equation}
is a $\mc P$-a.s.\ well defined $H^1$-valued with continuous trajectories and $\mc G_t$-progressively measurable. By Fubini's Theorem (see \cite[Thm.\ 4.18]{DaZ} for the stochastic case) a direct computation shows that $u^\delta$ solves \eqref{daprato5} in the sense that, $\mc P$-a.s.,
\begin{equation}
\label{daprato7}
u^\delta_t=w+\int_{t_0}^t\! \big[ A_\delta u^\delta_s+  F_\eta(u_s) \big] \, \rmd s +  \tilde{B} \alpha_t\;, \qquad t\in [t_0,t_1] \;.
\end{equation}
Observe that, for each $t \in [t_0,t_1]$,
\[
\mc E \Big(  \left\|  \int_{t_0}^t\! \big[S_\delta (t-s) -S(t-s) \big] \tilde{B} \, \rmd\alpha_s \right\|^2_{H^1} \Big)= \int_{t_0}^t\!\left\| \big[S_\delta (t-s) -S(t-s) \big] \tilde{B}\right\|^2_{HS} \rmd s \;.
\] 
Combining the previous identity with Lemma \ref{semigroups}, items (2) and (3), equation \eqref{daprato6} and dominated convergence easily imply that for each $t \in [t_0,t_1]$ $u^\delta_t \to u_t$ in $L^2(\Omega; H^1)$ and, in addition, $u^\delta \to u$ in $L^2(\Omega; L^2([t_0,t_1]\times \bb T^d))$ as $\delta \to 0$.
 
Let $\Phi^\delta \colon L^2 \to \bb R$ be defined by 
\[ 
\Phi^\delta(u) := \int\! \frac12| \nabla R_\delta u|^2 \, \rmd x = - \frac12 \langle u, R_\delta \Delta R_\delta u \rangle_{L^2} \, \;,
\]
so that for $u \in H^1$ we have $\Phi^\delta(u) \to \int\! \frac12| \nabla  u|^2 \, \rmd x =\Phi(u)$ as $\delta \to 0$. 

Since $\Phi^\delta$ is $C^2$ with locally bounded and uniformly continuous first and second derivatives, we can apply It\^o's  formula, see, e.g., \cite[Thm.\ 4.17]{DaZ}. Then, in view of \eqref{daprato5}, we get,
\begin{equation}
\label{daprato8}
\begin{split}
&\Phi^\delta(u^\delta_t)+ \int_{t_0}^t  \int\! \sigma(R_\delta v) \left| R_\delta \Delta R_\delta u^\delta_s \right|^2 \, \rmd x \, \rmd s \\ & \quad = \Phi^\delta(w) +\int_{t_0}^t\! \int R_\delta (-\Delta) R_\delta u^\delta_s F_\eta(u_s) \, \rmd x \, \rmd s \\ & \qquad + \frac{t-t_0}2 \, \mathrm{Tr}_{_{L^2}} \left( R_\delta (-\Delta) R_\delta B B^* \right) +\int_{t_0}^t\! \langle  R_\delta (-\Delta) R_\delta u^\delta_s,  B\,  \rmd\alpha_s  \rangle_{L^2}\;,
\end{split}
\end{equation}
where $B:=B(v)=\mathrm{Id}_{_{H^1 \to L^2}}\tilde{B}$. 

We next estimate separately the terms on the right-hand side of \eqref{daprato8}. By Young's inequality, for each $\zeta>0$ there exists $C_\zeta>0$ such that
\[
\begin{split}
& \int_{t_0}^t \int\! R_\delta (-\Delta) R_\delta u^\delta_s F_\eta(u_s) \, \rmd x \, \rmd s \le \zeta \int_{t_0}^t \int\! \left| R_\delta \Delta R_\delta u^\delta_s \right|^2 \, \rmd x \, \rmd s \\ & \qquad\qquad + C_\zeta \| \sigma\|_\infty \left(|f(0)|^2+\mathrm{Lip}(f)^2 \int_{t_0}^t\! \| u_s \|_{L^2}^2 \, \rmd s  \right)\; .
\end{split}
\] 
Clearly, if $\{ e_{\ell}\} \subset L^2$ is an orthonormal basis we have,
\[
\mathrm{Tr}_{_{L^2}} \left( R_\delta (-\Delta) R_\delta B B^* \right)=\sum_{\ell} \| R_\delta \nabla B e_{\ell}  \|_{L^2}^2\le C  \| \sqrt{\sigma( v)}\|_{H^2}^2 \| j \|_{H^1}^2 \;.
\]
Let $N_t$, $t \in [t_0,t_1]$, be the continuous martingale $ N_t=\int_{t_0}^t \langle  R_\delta (-\Delta) R_\delta u^\delta_s,  B\, \rmd\alpha_s  \rangle_{L^2} $. Since $B^* \psi= j* (\sqrt{2\sigma(v)} \psi)$, the quadratic variation of $N$ can be estimated as follows,
\[
[N]_t= \int_{t_0}^t\! \left\| B^* R_\delta (-\Delta) R_\delta u_s^\delta\right\|_{L^2}^2 \, \rmd s \le C  \| \sigma \|_\infty \| j\|^2_{L^1} \int_{t_0}^t\!  \int \left| R_\delta \Delta R_\delta u^\delta_s \right|^2 \, \rmd x \,\rmd s\;.
\]
By taking the supremum for $t \in [t_0, t_1]$ in \eqref{daprato8}, using again Young's inequality, taking the expectation, and gathering the above bounds together with the $L^2$-Doob's inequality,
\[
\mc E \Big(\sup_{t \in [t_0,t_1]} N_t \Big)^2 \le 4\, \mc E ([N]_{t_1}) \;,
\]
we conclude that there exists $C>0$ such that
\[
\begin{split}
& \mc E \left( \sup_{t \in [t_0,t_1]} \Phi^\delta(u^\delta_t) + \int_{t_0}^{t_1}\!  \int\! \left| R_\delta \Delta R_\delta u^\delta_s \right|^2 \, \rmd x \, \rmd s \right) \le C \bigg(  \mc E (\| w\|_{H^1}^2)   +  \| \sigma\|_\infty \|j\|_{L^1}^2 \\& + \| \sigma\|_\infty \Big( |f(0)|^2  +\mathrm{Lip}(f)^2 \int_{t_0}^{t_1}\!  \mc E (\| u_s \|_{L^2}^2) \, \rmd s \Big) +(t_1-t_0) \| \sqrt{\sigma( v)}\|_{H^2}^2 \| j \|_{H^1}^2 \bigg) \, .
\end{split}
\]
Since $\sup_{t \in [t_0,t_1]} \mc E (\| u_t\|_{H^1}^2) \le C (1+\mc E (\| w\|_{H^1}^2) )$ we finally get,
\begin{equation}
\label{daprato9}
\mc E \left( \sup_{t \in [t_0,t_1]} \Phi^\delta(u^\delta_t) + \int_{t_0}^{t_1}\! \int\!  \left| R_\delta \Delta R_\delta u^\delta_s \right|^2 \, \rmd x \, \rmd s \right) \le C \left( 1+ \mc E (\|w \|_{H^1}^2) \right) \; ,
\end{equation}
for some $C=C(t_1-t_0,\sigma, \|v\|_{H^2}, \mathrm{Lip}(f))>0$.

Since $u \in C([t_0,t_1];H^1)$, $\mc P$-a.s.\ we can take a countable dense set $\mc S \subset [t_0,t_1]$ and a subsequence still denoted by $\delta\to 0$ such that, $\mc P$-a.s., $u^\delta_s \to u_s$, $s\in \mc S$, in $H^1$ as $\delta \to 0$. Thus, as $\Phi^\delta \to \Phi$ pointwise in $H^1$, we get, $\mc P$-a.s., $\Phi(u_s) \le \varliminf_{\delta \to 0} \Phi^\delta(u^\delta_s)$ for all $s \in \mc S$. Then, the continuity $t \mapsto u_t$ implies that,  $\mc P$-a.s., $\sup_{t\in [t_0,t_1]} \Phi(u_t) \le \varliminf_{\delta \to 0} \sup_{t \in [t_0,t_1]} \, \Phi^\delta(u^\delta_t)$. By Fatou's' Lemma and \eqref{daprato9} we conclude $\mc E (\| u\|^2_{C([t_0,t_1];H^1)}) \le C\left( 1+ \mc E (\|w \|_{H^1}^2) \right)$. 

Again by Fatou's Lemma and \eqref{daprato9} we have, 
$$
\mc E \left(   \varliminf_{\delta \to 0} \| \Delta R_\delta R_\delta u^\delta \|^2_{L^2([t_0,t_1] \times \mathbb{T}^d)} \right)
\le C\left( 1+ \mc E (\|w \|_{H^1}^2) \right) \, .
$$
In particular, $\mc P$-a.s. we have $\varliminf_{\delta \to 0} \| \Delta R_\delta R_\delta u^\delta \|^2_{L^2([t_0,t_1] \times \mathbb{T}^d)}<\infty$. As $R_\delta R_\delta u^\delta \to u$ in $L^2(\Omega \times [t_0,t_1]\times \mathbb{T}^d)$, by elliptic regularity and lower semicontinuity we get $u \in L^2(\Omega \times [t_0,t_1] ;H^2)$ and $\mc E (\| u\|_{L^2(H^2)}^2) \le  C\left( 1+ \mc E (\|w \|_{H^1}^2) \right)$.  

Next, we show that $u$ satisfies \eqref{daprato2}. Fix $\psi \in C^\infty([t_0,t_1] \times \mathbb{T}^d)$ and consider the function $ \Psi\colon [t_0,t_1] \times L^2 \to \mathbb{R}$ given by $\Psi(t,u):=\langle \psi_t, u_t \rangle_{L^2}$. Clearly $\Psi$ is $C^2$ with locally uniformly continuous first and second derivatives, hence It\^o's  formula and \eqref{daprato5} give, for any $t\in [t_0,t_1]$,
\begin{equation}
\label{daprato10}
\begin{split}
\langle u^\delta_t, \psi_t\rangle_{L^2} & = \langle w, \psi_{t_0}\rangle_{L^2} + \int_{t_0}^t\! \Big[\langle u^\delta_s,\partial_s\psi_s\rangle_{L^2} + \langle A_\delta u^\delta_s +F_\eta(u_s), \psi_s \rangle_{L^2} \Big] \, \rmd s \\ & \quad + \int_{t_0}^t \langle \psi_s, B(v) \rmd \alpha_s\rangle_{L^2} \qquad \mbox{$\mc P$-a.s.}
\end{split}
\end{equation}
Recalling that $u^\delta \to u$ in $L^2(\Omega \times[t_0,t_1] \times \mathbb{T}^d)$ and $u^\delta_t \to u_t$ in $L^2(\Omega;L^2)$, up to subsequences we have $u^\delta \to u$ in $L^2([t_0,t_1] \times \mathbb{T}^d)$ and $u^\delta_t \to u_t$ in $L^2$ $\mc P$-a.s.. In order to take the limit as $\delta \to 0$ in \eqref{daprato10}, it remains to show that $\mc P$-a.s.\ we have $A_\delta u^\delta \rightharpoonup \sigma(v) \Delta u$ in $L^2([t_0,t_1] \times \mathbb{T}^d)$. To this end, notice that for $\mc P$-a.s.\ $\omega \in \Omega$ there exists a subsequence depending on $\omega$ such that $\Delta R_\delta R_\delta u^\delta \rightharpoonup \Delta u$ in $L^2([t_0,t_1] \times \mathbb{T}^d)$. Since $R_\delta v \to v$ in $H^2$, by Sobolev embedding we have $\sigma(R_\delta v) \to \sigma(v)$ uniformly, hence the desired statement follows.  

Finally, the bound \eqref{daprato3p} is the content of \cite[Thm.\ 7.4, item (iii)]{DaZ}. 
\end{proof}

\section{Existence of martingale solutions}
\label{sec:4}

In this section we prove the following existence result.

\begin{theorem}
\label{th:4.1}
Given $\bar u_0\in H^1$, there exists a martingale solution $\bb P$ to \eqref{1} with initial condition $\bar u_0$. Furthermore, for any $p\in [1,\infty)$ there exists $C>0$ such that
\begin{equation}
\label{a4terdue}
\bb E  \left(\sup_{t\in [0,T]} \| u_t\|^{2p}_{H^1} + \| u\|_{L^2(H^2)}^{2p} \right) \le C \;.
\end{equation}
\end{theorem}

The martingale solution will be obtained as a weak limit point of an approximating  sequence of probabilities on $C(L^2)$, that are the laws of a sequence of processes recursively defined according to the following scheme.

Let $(\Omega,\mc G,\mc G_t,\mc P)$ be a standard filtered probability space equipped with a $L^2$-cylindrical Wiener process $\alpha\colon \Omega\to C(H^{-\bar{s}})$, $\bar{s}>d/2$, such that $\mc G_t = \mc G^\alpha_t$ is the filtration generated by $\alpha$ completed with respect to $\mc P$. Given $\ell>0$ let also $W_\ell\colon\bb R \to\bb R$ be the $C^2$ function defined by
\be
\label{wl}
W_\ell(u):= \begin{cases} W(u) & \text{ if } |u|\le \ell\;, \\ W(\ell) + W'(\ell)(|u|-\ell) +\frac 12 W''(\ell)(|u|-\ell)^2 & \text { if } |u| >\ell\;. \end{cases}
\ee
Observe that, for any $\ell$ large enough, the function $W_\ell$ has quadratic growth at infinity both from above and below. Moreover, $W'_\ell$ is globally Lipschitz.

Fix $\eta>0$ and a smooth approximation $\imath_n$ of the Dirac $\delta$-function with $\|\imath_n\|_{L^1}=1$. Given $n\in \bb N$, consider the partition $0=t^n_0< t^n_1< \ldots < t^n_n=T$ with $t^n_{i+1}-t^n_i=T/n $ for $i=0, \ldots, n-1$. In each time step of this partition, we recursively construct a sequence of $\mathcal{G}_t$-adapted continuous processes $u^n$ and a sequence of $\mathcal{G}_t$-adapted processes $v^n$ which is constant on each time interval $[t^n_i, t^n_{i+1})$ as follows. Define, 
\begin{equation}
\label{a1.25}
v^n_t := \imath_n*\bar u_0 \in H^2(\mathbb{T}^d) \quad \hbox{for} \quad t\in[t^n_0,t^n_1)
\end{equation}
and set $\sigma^n_0=\sigma (v^n_{t^n_0})$. 
According to Proposition \ref{daprato} we define $u^n_t$,
$t\in [t^n_0,t^n_1)$, as a solution to 
\begin{equation}
\label{un0}
\begin{cases} 
\rmd u^n_t = \sigma^n_0 \big( \Delta u^n_t - R_\eta W_\ell'(R_\eta u^n_t)\big) \rmd t + \sqrt{2\sigma^n_0}\, j* \rmd \alpha_t\;, \\ u^n_0 = \bar{u}_0\;. 
\end{cases}
\end{equation}

By induction, for $1\le i \le n-1$ we define,
\begin{equation}
\label{a1.5}
v^n_{t} := \frac 1{t^n_{i}-t^n_{i-1}} \int_{t^n_{i-1}}^{t^n_{i}}\! u^n_s \, \rmd s, \qquad t\in [t^n_i,t^n_{i+1})
\end{equation}
and $\sigma^n_i:=\sigma(v^n_{t^n_i})$. Again by Proposition \ref{daprato}, we let $u^n_t$, $t\in [t^n_i,t^n_{i+1})$, be the solution to 
\begin{equation}
\label{uni}
\begin{cases} 
\rmd u^n_t = \sigma^n_i \big( \Delta u^n_t - R_\eta W_\ell'(R_\eta u^n_t)\big) \rmd t + \sqrt{2\sigma^n_i} \, j* \rmd \alpha_t\;, \\ u^n_{t^n_i}=\lim_{s \uparrow t^n_i} u^n_s \;.
\end{cases}
\end{equation}
We finally set $u^n_{t^n_n}=\lim_{s \uparrow t^n_n} u^n_s$. Notice that, by recursively using Proposition \ref{daprato} and \eqref{a1.5}, we get $\mc P$-a.s.\ $v^n \in L^\infty(H^2)$ and $u^n \in C(H^1) \cap L^2(H^2)$. Note that, although not indicated in the notation, the process $u^n$ also depends on $\eta$ and $\ell$.  

\smallskip
The proof of Theorem \ref{th:4.1} is split into three lemmata. 

\begin{lemma}[A priori bounds]
\label{lem:4.1}
Let $u^n$ be the process constructed by solving \eqref{un0}-\eqref{uni} and $\mathcal{F}_{\ell,\eta} \colon H^1 \to \mathbb{R}$ be the functional defined by
\begin{equation} 
\label{fle}
\mathcal{F}_{\ell,\eta}(u)= \int\! \frac12 |\nabla u|^2 + W_{\ell}(R_\eta u) \, \rmd x\;.
\end{equation} 
Then, for any $p\in[1,+\infty)$,
\begin{equation}
\label{a4bis}
\mc E  \left( \sup_{t\in [0,T]}\mc F_{\ell,\eta}(u^n_t) \right)^p + \frac12 \mc E  \left( \int_{0}^{T} \!\int\! \sigma(v^n_t) \big( \Delta u^n_t -R_\eta W'_\ell (R_\eta u^n_t)\big)^2\, \rmd x \, \rmd t \right)^p  \le C \, , 
\end{equation}
where $C>0$ depends only on $\|\bar{u}_0\|_{H^1}$, $T$, and $p$ but is independent of $n$, $\ell$, and $\eta$. 
\end{lemma}

\begin{proof}
The lemma is essentially achieved by applying It\^o's formula to $\mathcal{F}_{\ell,\eta}$, however we need to introduce a suitable approximation scheme to actually carry out the computation. Given $\delta>0$, let $\mathcal{F}^\delta_{\ell,\eta} \colon L^2 \to \mathbb{R}$ be defined by
\begin{equation}
\label{fle1}
\mathcal{F}^\delta_{\ell,\eta}(u)= \int\! \frac12 | \nabla R_\delta u|^2 + W_{\ell}(R_\eta u) \, \rmd x\;.
\end{equation}
 
By straightforward computations,  $\mathcal{F}^\delta_{\ell,\eta}$ is $C^2$ with locally bounded and uniformly continuous first derivative $\big(D \mathcal{F}^\delta_{\ell,\eta}\big)_u \in L^2$ and second derivative $\big(D^2  \mathcal{F}^\delta_{\ell,\eta} \big)_u \colon L^2 \to L^2$ given by
\[
\big(D \mathcal{F}^\delta_{\ell,\eta}\big)_u= R_\delta \Delta R_\delta u-R_\eta W'_\ell(R_\eta u) \; , \; \big(D^2  \mathcal{F}^\delta_{\ell,\eta} \big)_u=R_\delta (-\Delta) R_\delta + R_\eta W''_{\ell}(R_\eta u) R_\eta \; .
\]
Hence, by It\^o's formula, for each $t\in [t^n_{i},t^n_{i+1}]$ we have,
\begin{equation}
\label{a1}
\begin{split}
  & \mc F^\delta_{\ell,\eta}(u^n_t) + \int_{t^n_i}^{t} \!\int\!
  \sigma^n_i  \left( R_\delta \Delta R_\delta u^n_s-R_\eta W'_\ell(R_\eta u^n_s) \right)  \left( \Delta u^n_s -R_\eta W'_\ell(R_\eta u^n_s)\right)\, \rmd x \, \rmd s 
  \\
  &
  =  \mc F^\delta_{\ell,\eta}(u^n_{t^n_i})
  + \frac 12 \int_{t^n_i}^{t}\!  {\rm Tr}_{L^2} \left( B_{n,i}^* \left[ R_\delta (-\Delta) R_\delta + R_\eta W''_{\ell}(R_\eta u^n_s) R_\eta \right] B_{n,i}\right) \, \rmd s + N^{n,i, \delta}_t \, ,
\end{split}
\end{equation}
where $B_{n,i} \colon L^2 \to L^2$ and is defined as $B_{n,i} \psi= \sqrt{2 \sigma^n_i} j * \psi$ and $N^{n,i, \delta}$, $t\in [t^n_i,t^n_{i+1}]$, is the martingale 
\begin{equation}
N^{n,i, \delta}_t =\int_{t^n_i}^{t}   \left\langle  R_\delta (-\Delta) R_\delta u^n_s+R_\eta W'_\ell(R_\eta u^n_s), B_{n,i}\, \rmd\alpha_s \right\rangle_{L^2} \; .
\end{equation}

Letting $\{e_k\}$, $k \in  \mathbb{Z}^d$, be the standard orthonormal Fourier basis in $L^2$, we bound the trace terms as follows,
\[
\begin{split}
& {\rm Tr}_{L^2} \left( B_{n,i}^*  R_\delta (-\Delta) R_\delta  B_{n,i}\right)= \sum_k \|  R_\delta \nabla  B_{n,i} e_k \|^2_{L^2}\le \sum_k \| \nabla \left( \sqrt{2\sigma^n_i} j*e_k\right) \|^2_{L^2} \\ & \le 4 \sum_k  \left( \| (\nabla \sqrt{\sigma^n_i}) \, j*e_k  \|^2_{L^2}+\| \sqrt{\sigma^n_i}\, (\nabla j)* e_k \|^2_{L^2} \right) \\ & \le C \sum_k  \left( \frac{\|\sigma'\|_\infty^2}{4\inf \sigma} |\hat{j}(k)|^2  \| \nabla v^n_{t^n_i} \|^2_{L^2}+\| \sigma\|_\infty  | \hat{\nabla j}(k) |^2 \right) \le C(\sigma) \| j \|^2_{H^1} (1+\| \nabla v^n_{t^n_i} \|^2_{L^2})
\end{split} 
\]
and
\[
\begin{split}
& {\rm Tr}_{L^2} \left( B_{n,i}^*  R_\eta W''_{\ell}(R_\eta u^n_s) R_\eta B_{n,i}\right)= \sum_k \int\! |R_\eta B_{n,i} e_k |^2 W''_{\ell}(R_\eta u^n_s) \, \rmd x \\ & \quad \le  \sum_k \| R_\eta B_{n,i} e_k \|^2_{L^\infty} \int\! \left| W''_{\ell}(R_\eta u^n_s)\right| \, \rmd x \le \| \sigma\|_\infty \|j\|^2_{L^2} \int\! \left| W''_{\ell}(R_\eta u^n_s)\right| \, \rmd x\;, 
\end{split}
\]
where we used that $\| j*e_k\|_{L^\infty} \le |\hat{j}(k)|$ and $\| \nabla j*e_k\|_{L^\infty} \le |\hat{ \nabla j}(k)|$.

As $u^n \in L^2(\Omega;L^2([t^n_i, t^n_{i+1}];H^2))$ then $R_\delta (-\Delta) R_\delta u ^n -(-\Delta) u^n \to 0$ in $L^2(\Omega\times [t^n_i,t^n_{i+1}] \times \mathbb{T}^d)$. This implies $N^{n,i, \delta}_t \to N^{n,i}_t$ in $L^2(\Omega)$, where $N^{n,i}$ is the martingale
\begin{equation}
N^{n,i}_t =\int_{t^n_i}^{t}\! \left\langle  -\Delta u^n_s+R_\eta W'_\ell(R_\eta u^n_s), B_{n,i}\, \rmd\alpha_s \right\rangle_{L^2} \, .
\end{equation}
Indeed, 
\[ 
\mathcal{E}  \big( N^{n,i, \delta}_t -N^{n,i}_t \big)^2=\mathcal{E} \int^t_{t^n_i}\! \| B_{n,i} \left[ R_\delta (-\Delta) R_\delta u_s ^n -(-\Delta) u_s^n\right] \|^2_{L^2} \, \rmd s \overset{\delta \to 0}{\longrightarrow} 0\;.
\]
  
Using that $\mc P$-a.s.\ $u^n\in C([t^n_i,t^n_{i+1}]; H^1)\cap L^2([t^n_i,t^n_{i+1}]; H^2)$ we can take the limit as $\delta \to 0$ in the first three terms of \eqref{a1} by dominated convergence. As $N_t^{n,i,\delta} \to N_t^{n,i}$ $\mathcal{P}$-a.s. for a suitable subsequence, combining with the previous bound on the trace terms we finally get, for each $t\in [t^n_{i},t^n_{i+1}]$,
\begin{equation}
  \label{a2}
  \begin{split}
  & \mc F_{\ell,\eta}(u^n_t) + \int_{t^n_i}^{t}\!\int\!
  \sigma^n_i  \left(  \Delta  u^n_s-R_\eta W'_\ell(R_\eta u^n_s) \right)^2  \rmd x \, \rmd s 
  \\
  &
 \le  \mc F_{\ell,\eta}(u^n_{t^n_{i}})
  + C(\sigma)\| j \|^2_{H^1}
  \int_{t^n_i}^{t}\!  \int\!
   \left(1+| \nabla v^n_s |^2  +  \left| W''_{\ell}(R_\eta u^n_s)\right| \right)\, \rmd x
  \, \rmd s  +N^{n,i}_t \; ,
  \end{split}
  \end{equation}
where we used that $v^n$ is constant in the time interval $[t^n_i,t^n_{i+1})$.

Fix $t\in [0,T]$ and let $i_n(t)$ be such that $t\in[t^n_{i_n(t)},t^n_{i_n(t)+1})$, by summing \eqref{a2} in all the time intervals $[t^n_{j},t^n_{j+1})$, $j\le i_n(t)$, we deduce,
\begin{equation}
  \label{a2bis}
  \begin{split}
  & \mc F_{\ell,\eta}(u^n_t) + \int_{0}^{t} \!\int\!
  \sigma(v^n_s) \big( \Delta u^n_s -R_\eta W'_\ell (R_\eta u^n_s)\big)^2 \rmd x \, \rmd s 
  \\
  & \quad 
  \le\mc F_{\ell,\eta}(\bar u_{0})  
  + C(\sigma)\| j \|^2_{H^1}
  \int_{0}^{t}\!  \int\!
   \left(1+| \nabla v^n_s |^2  +  \left| W''_{\ell}(R_\eta u^n_s)\right| \right)\, \rmd x
  \, \rmd s + N^n_t\;,
  \end{split}
\end{equation}
where $N^n$ is the continuous $\mc P$-martingale $N^n_t=\sum_{j<i_n(t)} N^{n,j}_{t^n_{j+1}} +N^{n,i_n(t)}_t$.  In particular, the quadratic variation of $N^n$ is
\begin{equation}
\label{a3}
[N^n]_t = 2 \int_{0}^{t}\! \!\int\!\big\{ j*\big[ \sqrt{\sigma(v^n_s)}\big( -\Delta u^n_s + R_\eta W_\ell'(R_\eta u^n_s)\big)\big]\big\}^2 \, \rmd x \, \rmd s\;.
\end{equation}

By the assumptions on $W$ and the definition of $W_\ell$, \eqref{wl}, there exists $C>0$ independent of $\ell$ such that $|W''_\ell(\cdot)| \le C(1+W_\ell(\cdot))$. Moreover, for each $s\in [t^n_i,t^n_{i+1})$ with $i \ge 1$ we have, 
\[ 
\| \nabla v^n_s\|^2_{L^2} \le \frac{1}{t^n_i-t^n_{i-1}} \int^{t^n_i}_{t^n_{i-1}}\! \| \nabla u^n_{s'} \|^2_{L^2} \, \rmd s' \le \sup_{ s' \le s} \| \nabla u^n_{s'}\|^2_{L^2} \;. 
\]
Since $v^n_0 \equiv i_n * \bar{u}_0$, the previous bound yields $\| \nabla v^n_s\|_{L^2} \le \sup_{0 \le s' \le s} \| \nabla u^n_{s'}\|_{L^2}$ for any $0\le s \le T$. Thus, combining the two estimates above we have,
\[  
\int\! \left(1+| \nabla v^n_s |^2  +  \left| W''_{\ell}(R_\eta u^n_s)\right| \right)\, \rmd x \le C( 1+ \sup_{s' \le s} \mathcal{F}_{\ell,\eta}(u^n_{s'}))\;,
\]
hence, taking the supremum over time in \eqref{a2bis} we obtain,
\begin{equation}
\label{a4}
\begin{split}
  & \sup_{s\le t}\mc F_{\ell,\eta}(u^n_s) + \int_{0}^{t} \!\int\!
  \sigma(v^n_s) \big( \Delta u^n_s -R_\eta W'_\ell (R_\eta u^n_s)\big)^2 \rmd x \, \rmd s 
  \\
  & \qquad 
  \le 2\Big\{ \mc F_{\ell,\eta}(\bar u_{0})  
  + 
  C\int_{0}^{t}\! 1+\sup_{s' \le s} \mathcal{F}_{\ell,\eta}(u^n_{s'})
  \, \rmd s 
  + \sup_{s\le t}N^n_s \Big\}\: .
  \end{split}
  \end{equation}
Given $p \ge 1$, the previous inequality implies,
\begin{equation}
\label{a4p}
\begin{split}
  & \left( \sup_{s\le t}\mc F_{\ell,\eta}(u^n_s) \right)^p + \left(\int_{0}^{t} \!\int\!
  \sigma(v^n_s) \big( \Delta u^n_s -R_\eta W'_\ell (R_\eta u^n_s)\big)^2 \rmd x \, \rmd s \right)^p 
  \\
  & \qquad 
  \le C\Big\{  \left(\mc F_{\ell,\eta}(\bar u_{0}) \right)^p  
  + 
  \int_0^t\! 1+ \left(\sup_{s' \le s} \mathcal{F}_{\ell,\eta}(u^n_{s'}) \right)^p
  \, \rmd s 
  + \left(\sup_{s\le t}N^n_s \right)^p \Big\}\;,
  \end{split}
\end{equation}
for some $C=C_p>0$.

By Young's and BDG's inequalities (see, e.g., \cite{RY} for the latter) there exists a constant $C=C_p>0$ such that for each $\gamma>0$ we have,
\[ 
\begin{split}
& \mc E \left( \left(\sup_{s\le t}N^n_s \right)^p \right)  \le \frac\gamma 2  \mc E \left( \left(\sup_{s\le t}N^n_s \right)^{2p} \right)+ \frac{1}{2\gamma}\le 2 C \gamma \, \mc E ([N^n]^p_t ) + \frac{1}{2\gamma} \\ & \quad \le 4C\gamma \,    \mc E \left( \left(\int_{0}^{t} \!\int\! \sigma(v^n_s) \big( -\Delta u^n_s + R_\eta W_\ell'(R_\eta u^n_s)\big)^2 \,\rmd x \, \rmd s \right)^p \right)+\frac{1}{2\gamma}\;,
\end{split}
\]
where we used \eqref{a3}.  Choosing $\gamma>0$ small enough and taking the expectation in \eqref{a4p} we have,
  \begin{equation}
  \label{qen}
  \begin{split}
  & \mc E  \left(\sup_{s\le t}\mc F_{\ell,\eta}(u^n_s) \right)^p + \frac12 \mc E  \left(\left(\int_{0}^{t} \!\int\!
  \sigma(v^n_s) \big( \Delta u^n_s -R_\eta W'_\ell (R_\eta u^n_s)\big)^2 \rmd x \, \rmd s \right)^p \right)
  \\
  & \qquad 
  \le C \Big\{ \left( \mc F_{\ell,\eta}(\bar u_{0}) \right)^p  
  + 
  \int_{0}^{t}\!
   1+\mc E  \left(\sup_{s' \le s} \mathcal{F}_{\ell,\eta}(u^n_{s'})\right)^p \,\rmd s  \Big\}
   \end{split}
  \end{equation} 
As $\mc F_{\ell,\eta}(\cdot) \le C ( 1+ \| \cdot\|_{H^1}^4)$ for some $C$ independent of $\ell$ and $\eta$, applying recursively \eqref{daprato3} and \eqref{daprato3p} on each time interval $[t^n_i,t^n_{i+1}]$ we get $\mc E \left(\sup_{t\in [0,T]}\mc F_{\ell,\eta}(u^n_t) \right)^p  <\infty$. Thus, the bound \eqref{a4bis} follows from \eqref{qen} by Gronwall's inequality.
\end{proof}

\begin{lemma}[Tightness of the approximating sequence] 
\label{lem:4.2}
Let $\bb P^n_{\ell,\eta}$ be the law of the process $u^n$ constructed by solving \eqref{un0}-\eqref{uni}. Then $(\bb P^n_{\ell,\eta})$ is a tight family of probabilities on $C(L^2)$. 
\end{lemma}

\begin{proof}
In view of compact embedding $H^1 \hookrightarrow L^2$, a sufficient condition for a subset $A$ of $C(L^2)$ to be precompact is that
\be
\label{cc}
\sup_{u\in A} \sup_{0\le t \le T}  \| u_t\|_{H^1} < +\infty\;, \qquad  \lim_{\delta\to 0} \sup_{u\in A} \omega(u;\delta) = 0\;, 
\ee
where $\omega(u;\delta)$ is the modulus of continuity of the element $u\in C(L^2)$, i.e.,
\begin{equation}
\label{contmod}
\omega(u;\delta) := \sup_{\substack{t,s\in[0,T]\\ |t-s|\le \delta}} \| u_t-u_s\|_{L^2}\;.
\end{equation}
The family $(\bb P^n_{\ell,\eta})$ is tight if the following conditions are fulfilled.

(i) For each $\zeta>0$ there exists $a>0$ such that
\[
\bb P^n_{\ell,\eta}\Big( \sup_{0\le t \le T}  \| u_t\|^2_{H^1}>a\Big) \le \zeta \qquad \forall\, n,\ell,\eta\;.
\]

(ii) For each $\eps>0$ and $\zeta>0$ there exists $\delta\in (0,T)$ such that
\[
\bb P^n_{\ell,\eta}\big(\omega(u;\delta)>\eps\big) \le \zeta \qquad \forall\, n,\ell,\eta\;.
\]

Indeed, if (i) and (ii) are verified, given any $\zeta>0$ we can find $a>0$ and $\delta_k$, $k\in\bb N_0$, such that
\[
\bb P^n_{\ell,\eta}\Big( \sup_{0\le t \le T}  \| u_t\|^2_{H^1}\le a\Big) > 1 - \frac \zeta 2\;, \qquad 
\bb P^n_{\ell,\eta}\Big(\omega(u;\delta_k)\le \frac 1k\Big) > 1-\frac\zeta{2^{k+1}}\;.
\] 
Therefore, the closure $K_\zeta$ of the set
\[
\Big\{u\colon \sup_{0\le t \le T}  \| u_t\|^2_{H^1}\le a\;, \;\; \omega(u;\delta_k)\le \frac 1k\;\;\forall\,k\in \bb N_0\Big\}
\]
is compact in view of \eqref{cc} and has probability $\bb P^n_{\ell,\eta}(K_\zeta)>1-\zeta$. 

Now, we claim that, for any $p\ge 1$,
\begin{equation}
\label{a5}
\sup_{\ell,\eta, n} \mc E \Big( \sup_{0\le t \le T}  \| u^n_t\|^{2p}_{H^1} \Big) <\infty \, 
\end{equation}
and, for each $p>1$,
\begin{equation}
\label{a5b}
\lim_{\delta \to 0}\sup_{\ell,\eta, n} \sup_{s\in [0,T-\delta]} \frac 1 \delta \mc E \Big(\sup_{t\in [s,s+\delta]} \|u_t-u_s\|_{L^2}^{2p}\Big) =0 \;.
\end{equation}
By Chebyshev's inequality, \eqref{a5} implies (i) and, by a simple inclusion of events, see, e.g., \cite[Eq.\ (8.9)]{Bi}, and again Chebyshev's inequality, \eqref{a5b} implies (ii). 

The estimate \eqref{a5} is a direct consequence of \eqref{a4bis} since it can be easily checked that, in view of the assumptions on $W$, there exists $C>0$ such that $\| u\|^2_{H^1} \le C\big(1+\mc F_{\ell,\eta}(u) \big)$ for any $\ell,\eta$. 

To prove \eqref{a5b} we observe that by \eqref{un0}-\eqref{uni} and It\^o's formula, for each $s\in [0,T-\delta]$ and
$t\in [s,s+\delta]$,
\be
\label{rp1}
\| u^n_t -u^n_s\|_{L^2}^2 = A^{s,n}_t  + R^{s,n}_t + M^{s,n}_t\;,
\ee
where 
\[
A^{s,n}_t := 2\int_s^t\! \int\! \sigma(v^n_r) \big[\Delta u^n_r -R_\eta W_\ell'(R_\eta u^n_r)\big](u^n_r -u^n_s) \, \rmd x\, \rmd r \;,
\]
and 
\be
\label{rp2}
R^{s,n}_t = \int_s^t \! \int\! \big( j * \sqrt{\sigma(v^n_r)} \big)^2 \, \rmd x\, \rmd r \le \|\sigma\|_\infty \|j\|_{L^2}^2\, \delta\;;
\ee
finally, $M^{s,n,\phi}_t$, $t\in [s,s+\delta]$, is a continuous square integrable $\mc P$-martingale with quadratic variation,
\be
\label{rp}
\begin{split}
\big[ M^{s,n}\big]_t & = 4 \int_s^t \! \int  \big[ j *\big( \sqrt{\sigma(v^n_r)}(u^n_r -u^n_s)  \big)\big]^2 \, \rmd x\, \rmd r \\ & \le 4 \|\sigma\|_\infty \sup_{r\in [s,s+\delta]} \|u^n_r-u^n_s\|_{L^2}^2\, \delta\;,
\end{split}
\ee

By Cauchy-Schwartz inequality,
\[
\begin{split}
|A^{s,n}_t| & \le 2 \|\sigma\|_\infty \Big(\int_s^t \| u^n_r -u^n_s\|_{L^2}^2 \, \rmd r\Big)^{\frac 12} \\ & \qquad\qquad \times  \Big(\int_s^t\! \int\! \sigma(v^n_r) \big(\Delta u^n_r -R_\eta W_\ell'(R_\eta u^n_r)\big)^2 \, \rmd x\, \rmd r\Big)^{\frac 12}  \\ & \le 2\,\delta^{\frac 12} \, \|\sigma\|_\infty\sup_{r\in [s,s+\delta]} \|u^n_r-u^n_s\|_{L^2} \\ &  \qquad\qquad \times  \Big(\int_{0}^{T}\! \int\!\sigma(v^n_r) \big(\Delta u^n_r -R_\eta W'_\ell (R_\eta u^n_t)\big)^2 \rmd x\, \rmd r \Big)^{\frac 12}\;,
\end{split}
\]
so that, by Young's inequality, there exists $C>0$ such that
\[
|A^{s,n}_t| \le \frac 12 \sup_{r\in [s,s+\delta]} \|u^n_r-u^n_s\|_{L^2}^2 + C\delta\int_{0}^{T}\! \int\!\sigma(v^n_r) \big(\Delta u^n_r -R_\eta W'_\ell (R_\eta u^n_t)\big)^2 \rmd x\, \rmd r\;.
\]
Therefore, taking the supremum for $t\in [s,s+\delta]$ in \eqref{rp1} we deduce,
\be
\begin{split}
\label{rp1b}
\sup_{t\in[s,s+\delta]}\| u^n_t -u^n_s\|_{L^2}^2 & \le 2\sup_{t\in[s,s+\delta]}M^{s,n}_t  + 2\sup_{t\in[s,s+\delta]}R^{s,n}_t \\ & \quad + 2C\delta\int_{0}^{T}\! \int\!\sigma(v^n_r) \big(\Delta u^n_r -R_\eta W'_\ell (R_\eta u^n_t)\big)^2 \rmd x\, \rmd r\;.
\end{split}
\ee

By BDG inequality, see, e.g., \cite{RY}, for any $p>1$ there exists $C=C_p$ such that
\[
\mc E \Big(\sup_{t\in[s,s+\delta]} (M^{s,n}_t)^p \Big)  \le C \, \mc E\big( \big[ M^{s,n}\big]_{s+\delta}^{p/2}\big) \le \frac 1{2^{p+1}} \mc E \Big(\sup_{r\in [s,s+\delta]} \|u^n_r-u^n_s\|_{L^2}^{2p}\Big) + C \delta^p\;,
\]
where we used the bound \eqref{rp} and Young's inequality in the second step. By taking the $p$-th power and then the expectation value in \eqref{rp1b}, the last bound, together with \eqref{a4bis} and \eqref{rp2} implies the claim \eqref{a5b}. 
\end{proof}

\begin{lemma}[Properties of the cluster points]
\label{lem:4.3}
Let $\bb P$ be a cluster point of the sequence $(\bb P^n_{\ell,\eta})$. Then $\bb P$ is a martingale solution to \eqref{1} with initial condition $\bar u_0$. Furthermore, $\bb P$ satisfies the bound \eqref{a4terdue}.
\end{lemma}

\begin{proof}
Let $\bb P$ be a cluster point of the sequence $(\bb P^n_{\ell,\eta})$, so that, passing to a subsequence, $\bb P^n_{\ell,\eta} \to \bb P$ weakly.  

We start by proving the estimate \eqref{a4terdue}. In view of the assumptions on $W$ and the definition \eqref{wl} of $W_\ell$, there is $C>0$ such that $\| u\|^2_{H^1} \le C(1+\mc F_{\ell,\eta}(u))$ and, by Sobolev embedding, $\| R_\eta W'_{\ell} (R_\eta u)\|^2_{L^2} \le C (1+ \| u\|^6_{H^1})$ for any $u \in H^1$, where $C>0$ is independent of $\ell$ and $\eta$.

Hence, the bound \eqref{a4bis} combined with Calderon-Zygmund inequality readily implies, for any $p\ge 1$,
\begin{equation}
\label{a4ter}
\bb E^n_{\ell,\eta}  \left(\sup_{t\in [0,T]} \| u_t\|^{2p}_{H^1} + \| u\|_{L^2(H^2)}^{2p} \right) \le C \; , 
\end{equation}
where $C>0$ depends only on $\|\bar{u}_0\|_{H^1}$, $T$ and $p$ but is independent of $n$, $\ell$, and $\eta$. Since both the norms in \eqref{a4ter} are lower semicontinuous under $C(L^2)$-convergence, by Portmanteau's Theorem we infer that for any $p \ge 1$ the bound \eqref{a4terdue} holds with the same constant $C>0$ in \eqref{a4ter}.

Now, we show that $\bb P$ is a martingale solution to \eqref{1} with initial condition $\bar u_0$. By construction, $\bb P^n_{\ell,\eta}( u_0 =\bar u_0)=1$ for any $n, \ell, \eta$ so that $\bb P( u_0 =\bar u_0)=1$. Furthermore, by \eqref{a4terdue}, $\bb P (u \in L^\infty(H^1)\cap L^2(H^2) )=1$.

It remains to prove that for any $\psi\in C^\infty([0,T]\times \bb T^d)$ the process $M^\psi$ as defined in \eqref{a0} is a continuous square integrable $\bb P$-martingale with quadratic variation as in \eqref{a0.5}.

Fix $0\le s< t\leq T$, $\psi\in C^\infty([0,T]\times \bb T^d)$, and for $u \in L^\infty(H^1) \cap L^2(H^2)$ let
\begin{equation}
\label{gnle}
\begin{split}
G_{n,\ell,\eta}(u) & := \langle u_t,\psi_t\rangle_{L^2} - \langle u_s,\psi_s\rangle_{L^2}  \\ & \quad - \int_s^t \! \big\{\langle u_r, \partial_r \psi_r\rangle_{L^2} + \langle \sigma(v^n_r) \big[ \Delta u_r -R_\eta W'_\ell (R_\eta u_r)\big], \psi_r\rangle_{L^2} \big\}\, \rmd r\;,
\end{split}
\end{equation}
with $v^n$ the average of $u$ defined as in \eqref{a1.25}, \eqref{a1.5}. Similarly, let
\begin{equation}
\label{g}
\begin{split}
G(u) & := \langle u_t,\psi_t,\rangle_{L^2} - \langle u_s,\psi_s\rangle_{L^2}  \\ & \quad - \int_s^t \! \big\{ \langle u_r, \partial_r\psi_r \rangle_{L^2} + \langle \sigma(u_r) \big[ \Delta u_r - W' ( u_r)\big], \psi_r \rangle_{L^2}\big\}\, \rmd r\;,
\end{split}
\end{equation}
and, for $\delta>0$ and $u \in C(L^2)$, we define the regularized version of \eqref{g} as
\begin{equation}
\label{gdelta}
G^\delta(u)= G(R_\delta u) \;.
\end{equation}
Observe that the function defined in \eqref{gdelta} is continuous on $C(L^2)$ since $R_\delta: L^2 \to H^2$ continuously. Furthermore, by Sobolev and H\"older inequalities and since $R_\eta$ contracts any $L^p$ norm, 
\begin{equation}
\label{gnlegrowth}
|G_{n,\ell,\eta}(u)|+ |G(u)| + |G^\delta(u)|\le C(1+ \| u\|_{L^2(H^2)}+ \| u\|^3_{L^\infty(H^1)})\;,  
\end{equation}
where $C>0$ does not depend on $n, \ell, \eta$, and $\delta$. As a consequence, by \eqref{a4ter} and \eqref{a4terdue} we get, for any $p \ge 1$,
\begin{equation}
\label{attesap} 
\bb E^n_{\ell,\eta} \left( |G^\delta|^p+ |G_{n,\ell,\eta}|^p \right) +\bb E  \left(|G |^p+ |G^\delta|^p  \right) \le C\;,
\end{equation}
where $C>0$ does not depend on $n, \ell, \eta$, and $\delta$.

In view of \eqref{a1.25}-\eqref{uni}, for any $F\colon C(L^2)\to \bb R$ continuous, bounded, and measurable with respect to the canonical filtration at time $s$ we have,
\begin{equation}
\label{mp1n}
\bb E^n_{\ell,\eta}( F(u) G_{n,\ell,\eta}(u)) =0  
\end{equation} 
and
\begin{equation}
\label{mp2n}
\bb E^n_{\ell,\eta} \left( F(u)  \left[ G_{n,\ell,\eta}(u)^2 -\int_s^t\! \int\! \left( j*(  \sqrt{2 \sigma(v^n_r)} \psi_r) \right)^2 \,\rmd x\,\rmd r\right] \right)=0\;.
\end{equation} 
We would like to pass to the limit in \eqref{mp1n}-\eqref{mp2n} as $n , \ell \to +\infty$ and $\eta \to 0$ in order to conclude that
\begin{equation}
\label{mp1}
\bb E(F(u) G(u)) =0  
\end{equation} 
and
\begin{equation}
\label{mp2}
\bb E \left( F(u)  \left[ G(u)^2 -\int_s^t\!\int\! \left( j*(  \sqrt{2 \sigma(u_r)} \psi_r) \right)^2\,\rmd x\,\rmd r\right] \right)=0\;,
\end{equation} 
which, by the arbitrariness of $F$ and $0\le s< t \le T$, shows that $M^\psi$ as defined in \eqref{a0} is a continuous $\bb P$-martingale with quadratic variation as in \eqref{a0.5}. 

In order to prove \eqref{mp1} and \eqref{mp2} we use an approximation scheme based on \eqref{gdelta}. We fix a decreasing sequence $\delta_k \searrow 0$ and, for each $a>0$ we define $\mc D_a$ as the closure in $C(L^2)$ of the following set,
\begin{equation}
\label{seta}
\bigcap_{k\in\bb N}\Big\{  u \in L^\infty(H^1) \cap L^2(H^2) \colon \| u\|^2_{L^\infty(H^1)}+\| u\|^2_{L^2(H^2)} \le a \;,\;\;  \omega(u;\delta_k) < \frac 1k \Big\}\;,
\end{equation}
where $\omega$ is defined in \eqref{contmod}.
 
We observe that in view of \eqref{cc} the set $\mc D_a$ is compact for any $a>0$ . Moreover, by Lemma \ref{lem:4.1} and the argument in the proof of Lemma \ref{lem:4.2}, we can choose $\delta_k \searrow 0$ such that 
\begin{equation}
\label{sseta}
\lim_{a \to +\infty}  \sup_{n, \ell, \eta} \bb P^n_{\ell,\eta} (\mc D_a^\mathrm{c})=0 \;, \qquad \lim_{a \to +\infty}  \bb P (\mc D_a^\mathrm{c})=0 \;. 
\end{equation}
We claim that, for each $a>0$,
\begin{equation}
\label{unifconv}
\lim_{\delta \to 0} \varlimsup_{n, \ell ,\eta} \sup_{u \in \mc D_a} |G_{n,\ell,\eta}(u)-G^\delta(u)| =0 \;, \qquad  \lim_{\delta \to 0}  \sup_{u \in \mc D_a} |G(u)-G^\delta(u)| =0 \;.
\end{equation}
Postponing the proof of this claim we first derive \eqref{mp1}. We write,
\begin{equation}
\label{mp3}
\bb E^n_{\ell,\eta} (F G_{n,\ell,\eta})= \bb{E}^n_{ \ell,\eta} (F G^\delta) + \bb{E}^n_{\ell,\eta}( \id_{ \mc D_a} F(G_{n, \ell, \eta} -G^\delta)) + \bb{E}^n_{\ell,\eta}( \id_{ \mc D^\mathrm{c}_a} F(G_{n, \ell, \eta} -G^\delta))\;.
\end{equation}
Since $F$ is bounded, by \eqref{unifconv}, for any $a>0$,
\[ 
\lim_{\delta \to 0} \varlimsup_{n, \ell ,\eta} \bb{E}^n_{\ell,\eta}( \id_{ \mc D_a} |F(G_{n, \ell, \eta} -G^\delta)|)=0
\]
and, in view of \eqref{attesap}, \eqref{sseta}, and Chebyshev's inequality,
\[ 
\lim_{a\to\infty}  \varlimsup_{\delta \to 0} \varlimsup_{n, \ell ,\eta} \bb{E}^n_{\ell,\eta}( \id_{ \mc D^\mathrm{c}_a} |F(G_{n, \ell, \eta} -G^\delta)|)=0\;,
\]
hence, by \eqref{mp1n} and \eqref{mp3},
\[
\lim_{\delta \to 0} \lim_{n, \ell ,\eta}\bb{E}^n_{ \ell,\eta} (F(u) G^\delta(u)) =0 \;.
\]

Since $G^\delta \colon C(L^2) \to \mathbb{R}$ is continuous and satisfies \eqref{attesap} we get,
\[ 
0 = \lim_{\delta \to 0} \bb{E} (F G^\delta)=  \bb{E} (F G)+ \lim_{\delta \to 0} \bb E( F(G^\delta -G))\;.
\]
Finally, writing
\[ \bb E( F(G^\delta -G))= \bb{E}( \id_{ \mc D_a} F(G^\delta -G)) + \bb{E}( \id_{ \mc D^\mathrm{c}_a} F(G^\delta -G))\;,
\]
by using \eqref{attesap}, \eqref{sseta}, and \eqref{unifconv} as before we obtain \eqref{mp1}.

In order to prove \eqref{unifconv}, first notice that for $u \in C(L^2)$ and $v^n$ the average of $u$ defined as in \eqref{a1.25}, \eqref{a1.5} we have,
\begin{equation*}
\begin{split}
\| u-v^n\|_{L^\infty(L^2)} & \le \sup_{0 \le t<  \frac Tn} \|u_t- \imath_n* u_0 \|_{L^2} \vee \max_{i=1, \ldots, n-1} \sup_{t \in [t^n_i, t^n_{i+1})} \left\| u_t - \frac{n}T \int^{t^n_{i}}_{t^n_{i-1}}\! u_s \, \rmd s \right\|_{L^2} \\
& \le \omega(u; 2T/n)+ \| \imath_n *u_0 -u_0\|_{L^2} \;.
\end{split}
\end{equation*} 
Thus, by definition of $\mc D_a$, for each $a>0$,
\begin{equation}
\label{u-v}
\lim_{n \to \infty} \sup_{u \in \mc D_a} \| u-v^n\|_{L^\infty(L^2)} =0 \;.
\end{equation}

We define
\begin{equation}
\label{gle}
\begin{split}
G_{\ell,\eta}(u) & := \int\!u_t\psi_t \, \rmd x - \int\!u_s\psi_s\, \rmd x  \\ & \quad + \int_s^t\!\int\! \big\{ u_r \partial_r \psi_r + \sigma(u_r) \big[ \Delta u_r -R_\eta W'_\ell (R_\eta u_r)\big] \psi_r \big\} \, \rmd x\, \rmd r\; .
\end{split}
\end{equation}
Since $\sigma$ is Lipschitz and by Sobolev embedding $\| R_\eta W'_{\ell} (R_\eta u)\|^2_{L^2} \le C (1+ \| u\|^6_{H^1})$ for any $u \in H^1$, where $C>0$ is independent of $\ell$ and $\eta$, by definition of $\mc D_a$ and \eqref{u-v} we easily obtain,
\begin{equation}
\label{unifconvgle}
\begin{split}
& \varlimsup_{n, \ell ,\eta} \sup_{u \in \mc D_a} |G_{n,\ell,\eta}(u)-G_{\ell,\eta}(u)| \\ & \quad \le \varlimsup_{n, \ell ,\eta} \sup_{u \in \mc D_a} C \| u-v^n\|_{L^\infty(L^2)} (1+\|u\|_{L^2(H^2)}+ \| u\|^3_{L^\infty(H^1)}) =0 \;.
 \end{split}
\end{equation}

We are going to show that 
\begin{equation}
\label{unifconv2}
\varlimsup_{\ell ,\eta} \sup_{u \in \mc D_a} |G_{\ell,\eta}(u)-G(u)| =0 \;, \qquad  \lim_{\delta \to 0}  \sup_{u \in \mc D_a} |G(u)-G^\delta(u)| =0 \;,
\end{equation}
which clearly imply \eqref{unifconv} by \eqref{unifconvgle}.

Since $\psi$ is bounded and $R_\eta$ contracts also the $L^1$-norm we estimate,
\begin{equation}
\label{decgleta}
 \begin{split}
&|G_{\ell,\eta}(u)-G(u)| \le C \int_s^t \!\int\!\big| W'( u_r) -R_\eta W'_\ell (R_\eta u_r)\big| \, \rmd x\, \rmd r \, ,\\ &\quad \le C \int_s^t \!\int\!  \big| W'( u_r) -R_\eta W' ( u_r)\big| + \big|  W'( u_r) - W'_\ell ( u_r)\big|\, \rmd x\, \rmd r \\ & \qquad + C \int_s^t \!\int\! \big|  W_\ell'( u_r) - W'_\ell (R_\eta u_r)\big|  \, \rmd x\, \rmd r =I+I\!I+I\!I\!I\;.
\end{split}
\end{equation}
Notice that $\| R_\eta v-v \|^2_{L^2} \le  \eta \| v\|^2_{H^1}$ for any $v \in H^1$ and, by Sobolev and Holder inequalities, for any $v \in H^2$ we have, 
\[ \| W'(v)\|_{H^1} \le C   \| 1+|v|^3 \|_{H^1} \le C (1+ \|v\|^3_{H^1}+\|v\|_{H^2}\| v\|^2_{H^1} )\;, 
\]
hence,
\begin{equation}
\label{stimaI}
I\le C \int_s^t\! \|W'( u_r) -R_\eta W' ( u_r) \|_{L^2} \, \rmd r \le C \eta \int_s^t\! (1+ \|u_r\|^3_{H^1}+\|u_r\|_{H^2}\| u_r\|^2_{H^1} ) \, \rmd r\;.
\end{equation}
On the other hand, by the assumptions on $W$ and the definition \eqref{wl} of $W_\ell$, we have $|W'(u)|+|W'_\ell(u)| \le C (1+|u|^3)$ for a $C>0$ independent of $\ell$. Combining Cauchy-Schwartz, Sobolev, Chebyschev, and Young inequalities we get,
\begin{equation}
\label{stimaII}
\begin{split}
I\!I\le C \int_s^t\! \int_{|u_r|>l}\! (1+|u_r|^3) \, \rmd r \le C  \int_s^t\! (1+ \|u_r\|^3_{L^6} ) | \{ |u_r|>l\}|^{1/2} \, \rmd r \\ \le C  \int_s^t\! (1+ \|u_r\|^3_{H^1} ) \| u_r\|^{1/2}_{L^2} \ell^{-1/2} \, \rmd r \le C \ell^{-1/2} (1+ \| u\|^4_{L^\infty(H^1)}) \;. 
\end{split}
\end{equation}
Finally, noticing that $|W'_\ell(u)-W'_\ell(u')| \le C (1+|u|^2+|u'|^2)|u-u'|$ for an absolute constant $C>0$ independent of $\tau, \tau '$, and $\ell$, by Cauchy-Schwartz, Holder, and Sobolev inequalities, arguing as above we get,
\begin{equation}
\label{stimaIII}
\begin{split}
I\!I\!I & \le  C \int_s^t\! \int\! (1+|u_r|^2+|R_\eta u_r|^2) |u_r-R_\eta u_r| \,\rmd x \, \rmd r \\ & \le C  \int_s^t\! (1+ \|u_r\|^2_{L^4} )  \| u_r -R_\eta u_r \|_{L^2} \, \rmd r \\ & \le  C \eta^{1/2} \int_s^t\! (1+ \| u_r\|^2_{H^1}) \| u_r\|_{H^1} \, \rmd r \le C \eta^{1/2} (1+ \| u\|^3_{L^\infty(H^1)}) \;. 
\end{split}
\end{equation}
Combining \eqref{decgleta}-\eqref{stimaIII} the first claim in \eqref{unifconv2} follows. The proof of the second claim in  \eqref{unifconv2} is entirely similar. Indeed, first recall that $\sigma$ and $\psi$ are bounded and that $R_\delta$ and $\Delta$ commute on $H^2$. Thus, in view of \eqref{g} and \eqref{gdelta}, 
\begin{equation}
\label{decgdelta}
\begin{split}
& |G^\delta(u) -G(u)| \le C  \left( \| u_t -R_\delta u_t\|_{L^1} + \| u_s -R_\delta u_s\|_{L^1} + \int_s^t\! \| u_r-R_\delta u_r \|_{L^1} \, \rmd r \right) \\ & \quad\qquad + C \int_s^t\! \| W'(u_r)-W'(R_\delta u_r) \|_{L^1}  \, \rmd r \\ & \quad\qquad +\left| \int_s^t\! \int\!  \left( R_\delta \sigma( R_\delta u_r) \Delta u_r - \sigma(u_r) \Delta u_r \right) \psi_r   \, \rmd x \, \rmd r \right| =I'+I\!I'+I\!I\!I'\;.
\end{split}
\end{equation}
Since $\| R_\delta v-v \|^2_{L^2} \le  \delta \| v\|^2_{H^1}$ for any $v \in H^1$, we have,
\begin{equation}
\label{stimaI'}
I' \le C \| u -R_\delta u \|_{L^\infty (L^2)} \le C \delta^{1/2} \| u\|_{L^\infty (H^1)}\;, 
\end{equation} and 
arguing as in \eqref{stimaII} we also obtain,
\begin{equation}
\label{stimaII'}
I\!I' \le C \delta^{1/2} (1+ \| u\|^3_{L^\infty(H^1)})\;. 
\end{equation}
Combining \eqref{decgdelta}-\eqref{stimaII'}, the second claim in \eqref{unifconv2} follows once we prove that, for each $a>0$,
\begin{equation}
\label{unifconv3}
\lim_{\delta \to 0}  \sup_{u \in \mc D_a} \left| \int_s^t\! \langle R_\delta \sigma( R_\delta u_r) \Delta  u_r - \sigma(u_r) \Delta u_r, \psi_r \rangle_{L^2} \, \rmd r \right| =0 \;.
\end{equation}
We argue by contradiction and suppose that \eqref{unifconv3} fails. Then, there exists $a>0$, $\rho>0$, $\delta_k \to 0$, and a sequence $\{ u^{(k)} \} \subset \mc D_a$ such that, for each $k\ge 1$,
\begin{equation}
\label{contrhp}
\left| \int_s^t\! \langle R_{\delta_k} \sigma( R_{\delta_k} u^{k}_r) \Delta u^{k}_r - \sigma(u^{k}_r) \Delta u^{k}_r , \psi_r\rangle_{L^2} \rmd r \right| \ge \rho>0 \;.
\end{equation}
It is easy to check that $ \{ u^{k}\}\subset  \mc D_a \subset C(L^2)$ is equicontinuous and $\{u_t^{k}\} \subset H^1$ is equibounded; in view of the compact embedding $H^1 \hookrightarrow L^2$ we can apply the Ascoli-Arzel\`a theorem to infer that, up to subsequences, $u^{k} \to u \in C(L^2)$ as $k \to \infty$. Moreover, by standard lower semicontinuity argument is easy to check that $u \in \mc D_a$ and in addition $\Delta u^{k} \rightharpoonup \Delta u$ in $L^2([0,T] \times \bb T^d)$ as $k \to \infty$. 

Since  $u^{k} \to u$ in $C(L^2)$ and $\sigma$ is bounded and continuous we have $R_{\delta_k} u^{k} \to u$, $\sigma(u^{k}) \to \sigma(u)$, $\sigma(R_{\delta_k}u^{k}) \to   \sigma(u)$ and $R_{\delta_k}\sigma(R_{\delta_k}u^{k}) \to   \sigma(u)$ in $C(L^2)$ and in turn in $L^2([0,T] \times \bb T^d)$ as $k \to \infty$. As $\psi$ is bounded and smooth and $\Delta u^{k} \rightharpoonup \Delta u$ in $L^2([0,T] \times \bb T^d)$, as $k \to \infty$ we have,
\[  
\lim_{k \to \infty} \int_s^t\! \langle R_{\delta_k} \sigma( R_{\delta_k} u^{k}_r) \Delta u^{k}_r-  \sigma(u^{k}_r) \Delta u^{k}_r, \psi_r \rangle_{L^2}\, \rmd r =0 \;, 
\]
which contradicts \eqref{contrhp} and proves \eqref{unifconv3}. 

To deduce \eqref{mp2} from \eqref{mp2n} we first notice that
\begin{equation}
\label{mp4}
\begin{split}
\lim_{n ,\ell,\eta} \bb E^n_{\ell,\eta} \left( F(u)  \int_s^t\! \int\! \left( j*(  \sqrt{2 \sigma(v^n_r)} \psi_r) \right)^2 \, \rmd x \, \rmd r  \right) \\ =\bb E \left( F(u)  \int_s^t\!  \int\! \left( j*(  \sqrt{2 \sigma(u_r)} \psi_r) \right)^2 \, \rmd x \, \rmd r \right)\;.
\end{split}
\end{equation}
Indeed, as $u \mapsto F(u)  \int_s^t \int \big( j*(  \sqrt{2 \sigma(u_r)} \psi_r) \big)^2 \, \rmd x \, \rmd r$ is bounded and continuous, restricting the expectations to $\mc D_a$  and its complement the conclusion follows from \eqref{sseta} and \eqref{u-v}.

Finally, arguing as in the proof of \eqref{mp1} and using \eqref{attesap} for some $p>2$ we have,
\[ 
\lim_{n , \ell, \eta} \bb E^n_{\ell,\eta} \left(F(u)   G_{n,\ell,\eta}(u)^2 \right) =\bb E \left( F(u)   G(u)^2 \right)\;,
\]
which together with \eqref{mp2n} and \eqref{mp4} yields \eqref{mp2}.
\end{proof}

\section{Uniqueness results and strong existence}
\label{sec:5}

In this section we conclude the proof of Theorem \ref{t:eu}. To connect the notions of martingale and strong solutions we first introduce the notion of weak solution.

A pair $((\Omega,\mc G,\mc G_t,\mc P),(u,\alpha))$, where $(\Omega,\mc G,\mc G_t,\mc P)$ is a standard filtered probability space and $(u,\alpha)$ are $\mc G_t$-adapted processes, is a \emph{weak solution} to \eqref{1} with initial datum $\bar u_0$ iff 
\begin{itemize}
\item[i)] $\alpha\colon \Omega\to C(H^{-\bar{s}})$, $\bar{s}>d/2$, is a $L^2$-cylindrical Wiener process with respect to  $\mc G_t$, i.e., it is a $L^2$-cylindrical Wiener process and its increments $\alpha_t-\alpha_s$ are independent of $\mc G_s$ for $0\le s<t\in [0,T]$;
\item[ii)] $u\colon \Omega \to C(L^2)$, $\mc P (u_0=\bar u_0) =1$, and $\mc P(u \in L^\infty(H^1) \cap L^2(H^2)) =1$;
\item[iii)] for each $\psi\in C^\infty\big([0,T]\times \bb T^d\big)$ and $t\in [0,T]$, the identity \eqref{form} holds $\mc P$-a.s.
\end{itemize}

\emph{Pathwise uniqueness} of weak solutions holds if whenever $((\Omega,\mc G,\mc G_t,\mc P),(u,\alpha))$ and $((\Omega,\mc G,\mc G_t,\mc P),(u',\alpha'))$ are two weak solutions on the same filtered space with $\alpha=\alpha'$ then $\mc P(u_t=u'_t \; \;\forall\, t\in [0,T])=1$.

We remark that if a weak solution $((\Omega,\mc G,\mc G_t,\mc P),(u,\alpha))$ is such that $u$ is $\mc G^\alpha_t$-adapted (recall that $\mc G^\alpha_t$ denotes the filtration generated by $\alpha$ completed with respect to $\mc P$) then the map $u\colon \Omega \to C(L^2)$ is a strong solution on the probability space $(\Omega,\mc G,\mc P)$ equipped with the cylindrical Wiener process $\alpha$. 

By a martingale representation lemma, we first show that existence of weak solutions can be deduced from the existence of martingale solutions. 

\begin{lemma}
\label{lem:5.1}
Given $\bar u_0\in H^1$, let $\bb P$ be a martingale solution to \eqref{1} with initial condition $\bar u_0$. There exists a  weak solution $((\Omega,\mc G,\mc G_t,\mc P),(u,\alpha))$ to \eqref{1} such that $\mc P \circ u^{-1} =\bb P$.
\end{lemma}

\begin{proof}
Let $\bb P$ be a martingale solution. Recall \eqref{a0} and let $\{e_k\}$, $k \in  \mathbb{Z}^d$, be an orthonormal basis in $L^2$. We claim that the process $M=(M_t)_{t\in[0,T]}$ defined by $M_t := \sum_{k} M^{e_k}_t e_k$ is a $L_2$-valued, continuous square integrable $\bb P$-martingale with quadratic variation,
\be
\label{Mvrt}
[M]_t = \int_0^t\!  B(u_s) B(u_s)^*\,\rmd s\;,
\ee
where we recall $B(u) \colon L^2 \to L^2$ is the  Hilbert-Schmidt operator given by $B(u)\psi = \sqrt{2 \sigma(u)} j *  \psi$. 

The martingale property of $M$ is obvious and \eqref{Mvrt} is a direct consequence of \eqref{a0.5}. Since $\| B(u)\|_{HS}$ is bounded uniformly w.r.t.\ $u \in L^2$, $M$ is square integrable. Moreover, by \eqref{a0},
\[
M_t = u_t - u_0 - \int_0^t\! \sigma(u_s)\Big(\Delta u_s - W'(u_s)\Big)\,\rmd s \quad\forall\,t\in [0,T] \qquad\bb P\mbox{-a.s.}\;,
\]
where the identity has to be understood between elements of $L^2$. Since $\bb P(u \in L^\infty(H^1) \cap L^2(H^2))=1$ we deduce the $\bb P$-a.s.\ continuity of $M$. 

In view of the previous claim, we can apply the representation theorem \cite[Thm.\ 8.2]{DaZ} and deduce the existence of an  enlargement of the filtered probability space $(C(L^2),\mc B, \mc B_t, \bb P)$, denoted by $(\Omega,\mc G,\mc G_t,\mc P)$, equipped with a cylindrical Wiener process $\alpha\colon \Omega\to C(H^{-\bar{s}})$, $\bar{s}>d/2$, and a $\mc G_t$-progressively measurable map $u\colon \Omega \to C(L^2)$  such that $\mc P \circ u^{-1} =\bb P$ and $M_t = \int_0^t \! B(u_s)\,\rmd \alpha_s$. In particular, $((\Omega,\mc G,\mc G_t,\mc P),(u,\alpha))$ is a weak solution to \eqref{1} with initial condition $\bar u_0$.  
\end{proof}

By the previous lemma and It\^o formula we next show the continuity property of the trajectories for martingale solutions.

\begin{lemma}
\label{lem:5.2}
Given $\bar u_0\in H^1$, let $\bb P$ be a martingale solution to \eqref{1} with initial condition $\bar u_0$. Then, $\bb P (u \in C(H^1))=1$.
\end{lemma}

\begin{proof}
Since $\bb P (u \in C(L^2)\cap L^\infty(H^1))=1$, we already know that $\bb P$-a.s.\ the trajectories are $H^1$-weak continuous, so we have only to show the $\bb P$-a.s.\ continuity of the real-valued process $t\mapsto \|u_t\|_{H^1}$. To this end, let $((\Omega,\mc G,\mc G_t,\mc P),(u,\alpha))$ be the weak solution associated to $\bb P$ as constructed in Lemma \ref{lem:5.1}. We shall  prove the $\mc P$-a.s.\ continuity of the map $t\mapsto \mc F(u_t)$, where $\mc F \colon H^1 \to \bb R$ is the functional,
\begin{equation} 
\label{fe}
\mc F(u)= \int\! \frac12 |\nabla u|^2 + W(u) \, \rmd x\;.
\end{equation}
Note indeed the map $[0,T] \ni t \mapsto \int\! W(u_t) \,\rmd x$ is $\mc P$-a.s.\ continuous since $\mc P(u\in C(L^2)\cap L^\infty(H^1))=1$.
 
In order to apply It\^o's formula to $\mc F$, we proceed by approximation as in the proof of  Lemma \ref{lem:4.1}. Given $\delta>0$ and $\ell>0$, let $\mc F^\delta_\ell \colon L^2 \to \mathbb{R}$ be the regularized version of $\mc F$ defined by
\begin{equation}
\label{fe1}
\mc F^\delta_\ell(u)= \int\! \frac12 | \nabla R_\delta u|^2 + W_{\ell}(R_\delta u) \, \rmd x\;,
\end{equation}
where, as usual, $R_\delta = (\mathrm{Id}-\delta\Delta)^{-1}$ and $W_\ell$ is defined in \eqref{wl}. Since $\mc F^\delta_\ell$ is $C^2$ with locally uniformly continuous first and second derivatives, we can apply It\^o's formula and deduce,
\[
\begin{split}
  & \mc F^\delta_\ell(u_t) + \int_0^t\!\int\!
  \sigma(u_s)  (\Delta u_s -W'(u_s))( R_\delta \Delta R_\delta u^n_s-R_\eta W'_\ell(R_\delta u^n_s))\, \rmd x \, \rmd s 
  \\
  & \quad 
  =  \mc F^\delta_\ell(u_0)
  + \frac 12 \int_0^t\!  {\rm Tr}_{L^2} (B(u_s)^* \left[ R_\delta (-\Delta) R_\delta + R_\delta W''_{\ell}(R_\delta u_s) R_\delta \right] B(u_s)) \, \rmd s \\ & \qquad + \int_0^t  \left\langle  R_\delta (-\Delta) R_\delta u_s+R_\delta W'_\ell(R_\delta u_s), B(u_s)\, \rmd\alpha_s \right\rangle_{L^2}\quad\forall\,t\in [0,T] \qquad\mc P\mbox{-a.s.} \;.
\end{split}
\]
Recall that, by definition of martingale solution, $\bb P(u \in L^\infty(H^1) \cap L^2(H^2))=1$. 
Given $\kappa >0$ let
\[
\tau_\kappa := \inf \bigg\{t\in [0,T] \colon  \int_0^t \! \|u_s\|_{H^2}^2 \, \rmd s >\kappa\bigg\} \;,
\]
setting $\tau_\kappa=T$ if the set in the right-hand side is empty. Note that $\tau_\kappa \uparrow T$ as $\kappa\to\infty$ $\bb P$-a.s.. By stopping at $\tau_\kappa$ the It\^o 's formula above, straightforward estimates (similar to those in the proof of  Lemma \ref{lem:4.1}) allow to take the limit $\delta\to 0$ and $\ell\to \infty$. By taking afterwards the limit as $\kappa \to \infty$ we finally get,
\[
\begin{split}
  & \mc F(u_t) + \int_0^t\!\int\!\sigma(u_s)  (\Delta u_s -W'(u_s))^2\, \rmd x \, \rmd s \\ & \quad =  \mc F(u_0) + \frac 12 \int_0^t\! {\rm Tr}_{L^2} (B(u_s)^* [-\Delta + W''(u_s)] B(u_s)) \, \rmd s \\ & \qquad  +  \int_0^t \!\left\langle -\Delta u_s + W'(u_s), B(u_s)\, \rmd\alpha_s \right\rangle_{L^2} \quad\forall\,t\in [0,T] \qquad\mc P\mbox{-a.s.} \;.
  \end{split}
\]
Since $\bb P(u \in L^\infty(H^1) \cap L^2(H^2))=1$, $\mc P$-a.s.\ the second term on the left-hand side is $\mc P$-a.s.\ continuous and the trace in the second term on the right-hand side is $\mc P$-a.s.\ bounded uniformly in time. Indeed, arguing as in \eqref{a1}-\eqref{a2bis} we deduce the inequality,
\[ 
\begin{split}
{\rm Tr}_{L^2} (B(u_s)^* [-\Delta + W''(u_s)] B(u_s)) &\leq C_\sigma \|j \|^2_{H^1} \int\! (1+|\nabla u_s|^2+|W''(u_s)|) \, \rmd x \\ & \leq C_\sigma \|j \|^2_{H^1} (1+ \mc F(u_s))\; . 
\end{split}
\]
 Combining these facts with the a.s.\ continuity of the stochastic integral we get the $\mc P$-a.s.\ continuity of the map $t\mapsto \mc F(u_t)$.
\end{proof}

By means of an $H^{-1}$ estimate inspired by \cite{AL}, we next prove pathwise uniqueness of weak solutions.

\begin{proposition}
\label{prop:5.1}
Let $((\Omega,\mc G,\mc G_t,\mc P),(u,\alpha))$ and $((\Omega,\mc G,\mc G_t,\mc P),(v,\alpha'))$ be two weak solutions to \eqref{1} with initial condition $\bar u_0\in H^1$ defined on the same filtered space. If $\alpha'=\alpha$ then $\mc P(u_t=v_t \; \forall\, t\in [0,T]) = 1$. 
\end{proposition}

\begin{proof}
We observe that by $\mc P$-a.s.\ continuity it is enough to show that $\mc P(u_t=v_t)=1$ for any $t\in [0,T]$. For $\kappa > \kappa_0 := 2\|\bar u_0\|_{H^1}$ we introduce the stopping time $\tau_\kappa\colon \Omega\to \bb R_+$ defined by
\[
\tau_\kappa := \inf\{t\in [0,T] \colon \|u_t\|_{H^1}+\|v_t\|_{H^1}>\kappa\}\;,
\]
setting $\tau_\kappa=T$ if the set in the right-hand side is empty. Note that, $\mc P$-a.s.\  the $H^1$-continuity yields  $\tau_\kappa >0$ and $\tau_\kappa \uparrow T$ as $\kappa \to \infty$. Therefore, it is enough to prove that $\mc P(u_{t\wedge \tau_\kappa} =v_{t\wedge \tau_\kappa})=1$ for any $\kappa>\kappa_0$ and $t\in [0,T]$. This will be achieved by showing that the real random variable $\Psi_t\colon\Omega\to \bb R$ defined by
\[
\Psi_t := \frac 12 \big\|h(u_{t\wedge \tau_\kappa}) - h(v_{t\wedge \tau_\kappa})\big\|_{H^{-1}}^2\;, \qquad h(u) := \int_0^u\!\frac{1}{\sigma(r)}\,\rmd r\;,
\]
is $\mc P$-a.s.\ vanishing for any $\kappa>\kappa_0$ and $t\in [0,T]$. To this purpose, we estimate the evolution in time of $\Psi_t$ via It\^o's calculus.

Since the function $\Psi\colon L^2\times L^2 \to \bb R$ defined by $\Psi(u,v):=\frac 12 \big\|h(u) - h(v)\big\|_{H^{-1}}^2$ is not twice differentiable, we proceed by approximation. We introduce the regularized version of $\Psi$ defined by $\Psi^\delta(u,v):=\frac 12 \big\|h(R_\delta u) - h(R_\delta v)\big\|_{H^{-1}}^2$. Since $h \colon \bb R \to \bb R$ is $C^3$, $R_\delta\colon L^2\to H^2$ is bounded, and $H^2$ is compactly embedded in $C(\bb T^d)$, it is easy to show that $\Psi^\delta$ is $C^2$. Setting $f^\delta(u,v) := R_1(h(R_\delta u)-h(R_\delta v))$, the first derivative $(D\Psi^\delta)_{u,v} \in L^2 \times L^2$ is given by 
\[
(D\Psi^\delta)_{u,v} = R_\delta \begin{pmatrix} h'(R_\delta u) f^\delta(u,v) \\ h'(R_\delta v) f^\delta(u,v) \end{pmatrix}\;, 
\]
and the second derivative $(D^2\Psi^\delta)_{u,v}\colon L^2\times L^2 \to L^2\times L^2$ reads,
\[
\begin{split}
(D^2\Psi^\delta)_{u,v} & = R_\delta \begin{pmatrix} h''(R_\delta u) f^\delta(u,v) & 0 \\ 0 & - h''(R_\delta v) f^\delta(u,v) \end{pmatrix} R_\delta \\ & \quad + R_\delta \begin{pmatrix} h'(R_\delta u) R_1 h'(R_\delta u) & - h'(R_\delta u) R_1 h'(R_\delta v) \\ - h'(R_\delta v) R_1 h'(R_\delta u) & h'(R_\delta v) R_1 h'(R_\delta v)\end{pmatrix} R_\delta\;.
\end{split}
\]

As $h$ is Lipschitz we have $ \| (D \Psi^\delta)_{u,v} \|_{L^2 \times L^2} \leq C  (\| u\|_{L^2}+\| v\|_{L^2})$ \, 
hence the first derivative is bounded on bounded subsets of $L^2 \times L^2$.
 As $R_\delta \colon L^2 \hookrightarrow C(\bb T^d)$ is compact, $L^2 \ni u \mapsto h'(R_\delta u) \in C(\bb T^d)$ is uniformly continuous on bounded subset, hence the uniform continuity of $D \Psi^\delta$ on bounded subset follows. Concerning the second derivative, notice that since $h'$ and $h''$ are bounded
 and $\| f^\delta(u,v) \|_{L^\infty} \leq C (\| u\|_{L^2}+\| v\|_{L^2})$ with a constant independent of $\delta$, for each $\phi_1 , \phi_2 \in L^2$ we have
 \begin{equation}
 \label{d2}
 \left|  \left\langle
 \begin{pmatrix} \phi_1 \\ \phi_2 \end{pmatrix} ,
  (D^2\Psi^\delta)_{u,v} \begin{pmatrix} \phi_1 \\ \phi_2 \end{pmatrix}
 \right\rangle_{L^2 \times L^2} \right| \leq C (\| u\|_{L^2}+\| v\|_{L^2}+1) (\| \phi_1 \|_{L^2}^2 + \| \phi_2\|_{L^2}^2) \, ,
 \end{equation}
where the constant $C$ does not depend on $\delta$, $\phi_1$, and $\phi_2$. The same argument used for the first derivative entails that the second derivative $D^2 \Psi^\delta$ is uniformly continuous on bounded subsets of $L^2 \times L^2$.

By It\^o's formula (notice $\Psi^\delta(u_0,v_0)=\Psi^\delta(\bar u_0,\bar u_0)=0$), 
\begin{equation}
\label{itodelta}
\begin{split}
\Psi^\delta(u_t,v_t) & = \int_0^t\! \bigg\langle(D\Psi^\delta)_{u_s,v_s}\,,\begin{pmatrix}\sigma(u_s)(\Delta u_s-W'(u_s)) \\ \sigma(v_s)(\Delta v_s-W'(v_s)) \end{pmatrix}\bigg\rangle_{L^2\times L^2}\,\rmd s \\ & \quad + \int_0^t\! \mathrm{Tr}_{L^2\times L^2} \big((D^2\Psi^\delta)_{u_s,v_s} \bb B(u_s,v_s) \bb B(u_s,v_s)^* \big)\,\rmd s \\ & \quad + \int_0^t\! \big\langle(D\Psi^\delta)_{u_s,v_s}\,, \bb B(u_s,v_s) \,\rmd \alpha_s \big\rangle_{L^2\times L^2}\;,
\end{split}
\end{equation}
where $\bb B$ is the Hilbert-Schmidt operator on $L^2\times L^2$ defined by 
\[
\bb B(u_s,v_s) =\begin{pmatrix}B(u_s) \\ B(v_s) \end{pmatrix}\;.
\]

In order to take the limit as $\delta \to 0$ in the previous identinty, notice that for $u ,v \in C(L^2)$ we have $\Psi^\delta (u_t,v_t) \to  \Psi(u_t,v_t)$ and 
\[
 \lim_{\delta \to 0} (D \Psi^\delta)_{u_t,v_t}  =   \begin{pmatrix} h'( u_t) R_1(h(u_t)-h(v_t)) \\ h'( v_t)  R_1(h(u_t)-h(v_t)) \end{pmatrix}\;, 
 \]
 in $L^2 \times L^2$ uniformly for $t \in [0,T]$, as $h, h'$ are Lipschitz and $R_1 \colon L^2 \hookrightarrow L^\infty$. Since $\mc P$-a.s.\ $u, v \in C(H^1) \cap L^2(H^2)$, this allows to pass to the limit also in the first term on the r.h.s. of \eqref{itodelta} by dominated convergence. Moreover, the same uniform convergence together with the computation of the quadratic variation allows to pass to the limit (up to subsequences) in the stochastic integral. Finally, we rewrite the trace term in \eqref{itodelta} as
\[
\begin{split}
& \mathrm{Tr}_{L^2\times L^2} \big((D^2\Psi^\delta)_{u_s,v_s} \bb B(u_s,v_s) \bb B(u_s,v_s)^* \big) \\ & \quad\qquad = \sum_k \big\langle  \bb B(u_s,v_s) e_k , (D^2 \Psi^\delta)_{u_s,v_s} \bb B(u_s,v_s) e_k \big\rangle_{L^2 \times L^2} \; , 
\end{split}
\]
where $\{ e_k \} \subset L^2$ is an orthonormal basis.  Since $\bb B$ is Hilbert-Schmidt and the bound \eqref{d2} holds, in order to take the limit in the trace term by dominated convergence w.r.t.\ to $k$ it is enough to show that 
\begin{equation}
\label{tracelimit}
\begin{split}
\lim_{\delta \to 0}
(D^2\Psi^\delta)_{u_s,v_s} &=  \begin{pmatrix} h''( u_s) R_1(h(u_s)-h(v_s)) & 0 \\ 0 &  -h''( v_s) R_1(h(u_s)-h(v_s)) \end{pmatrix}  \\ & \quad +  \begin{pmatrix} h'( u_s) R_1 h'( u_s) & - h'( u_s) R_1 h'( v_s) \\ - h'(v_s) R_1 h'( u_s) & h'( v_s) R_1 h'( v_s)\end{pmatrix} \;, 
\end{split}
\end{equation}
in the weak operator topology on $L^2 \times L^2$, uniformly for $s \in [0, T]$, as $h$, $h'$, $ h''$ are Lipschitz and $R_1 \colon L^2 \hookrightarrow L^\infty$.

By stopping at $\tau_\kappa$ and recalling that $h'  \equiv 1/\sigma $, we finally get
\[
\begin{split}
\Psi_t & = \int_0^{t\wedge \tau_\kappa}\! \big\langle R_1 (h(u_s)-h(v_s)), \Delta (u_s - v_s) - W'(u_s) + W'(v_s)\big\rangle_{L^2} \,\rmd s \\ & + \frac 12 \int_0^{t\wedge \tau_\kappa}\! \mathrm{Tr}_{L^2}\Big(\big[R_1 (h(u_s)-h(v_s)) \big]\big[h''(u_s) B(u_s)B(u_s)^* - h''(v_s) B(v_s)B(v_s)^*\big] \\ & \qquad + (B(u_s)^*h'(u_s) - B(v_s)^*h'(v_s)) R_1 (h'(u_s) B(u_s) - h'(v_s) B(v_s)) \Big) \, \rmd s \\ & + \int_0^{t\wedge \tau_\kappa}\! \Big\langle R_1 [h(u_s)-h(v_s)], \Big[\frac{1}{\sigma(u_s)} B(u_s) - \frac{1}{\sigma(v_s)} B(v_s)\Big] \, \rmd \alpha_s \Big\rangle_{L^2}\;.
\end{split}
\]
Let $\{ e_k \} \subset L^2$ be the Fourier orthonormal basis and define

\[
f_s = R_1 (h(u_s)-h(v_s))\;, \qquad \beta_s  := \frac1{\sqrt{\sigma(u_s)}} - \frac1{\sqrt{\sigma(v_s)}}\;.
\]
By using that $R_1\Delta = -\mathrm{Id} + R_1$, and recalling the definitions of $h(u)$ and $B(u)$, after some simple algebraic computations evaluating the trace by Fourier series the above identity reads,
\be
\label{dt}
\Psi_t + \int_0^{t\wedge \tau_\kappa}\! \big\langle  h(u_s)-h(v_s), u_s - v_s\big\rangle_{L^2} \,\rmd s = I^1_t+I^2_t+I^3_t+  M_{t\wedge\tau_\kappa}\;,
\ee
where
\[
\begin{split}
I^1_t & := \int_0^{t\wedge \tau_\kappa}\! \big\langle f_s, u_s - v_s - W'(u_s) + W'(v_s)\big\rangle_{L^2} \,\rmd s \;, \\ I^2_t & := \| j \|_{L^2}^2\int_0^{t\wedge \tau_\kappa}\! \int\! f_s(x)  \left(\frac{\sigma'(v_s(x))}{\sigma(v_s(x))} - \frac{\sigma'(u_s(x))}{\sigma(u_s(x))}\right) \, \rmd x\;, \\ I^3_t & := \int_0^{t\wedge \tau_\kappa}\! \sum_{k\in\bb Z^d} \big\langle \beta_s j * e_k , R_1 \beta_s j * e_k \big\rangle_{L^2}\;,
\end{split}
\]
and $M_{t\wedge\tau_\kappa}$ is the stochastic integral (the last term in the previous It\^o's formula).  

By Assumption \ref{t:ws}, H\"older inequality, and the Sobolev embedding $H^1 \hookrightarrow L^6$ we have,
\[
\begin{split}
|I^1_t| & \le \int_0^{t\wedge \tau_\kappa}\! \big\langle |f_s|, |u_s - v_s|(1+ u_s^2 + v_s^2) \big\rangle_{L^2} \,\rmd s \\ & \le \int_0^{t\wedge \tau_\kappa}\! \|f_s\|_{L^6}  \|u_s - v_s\|_{L^2} \|1+ u_s^2 + v_s^2\|_{L^3} \,\rmd s \\ & \le C \int_0^{t\wedge \tau_\kappa}\! \|f_s\|_{H^1}  \|u_s - v_s\|_{L^2} \big(1+ \|u_s\|_{H^1}^2 + \|v_s\|_{H^1}^2\big) \,\rmd s   \\ & \le C (1+\kappa^2)  \int_0^{t\wedge \tau_\kappa}\! \sqrt{\Psi_s}\, \|u_s - v_s\|_{L^2} \,\rmd s \;,
\end{split}
\]
where in the last inequality we used that $\|f_s\|_{H^1}=\sqrt{2\Psi_s}$ and $\|u_s\|_{H^1}^2 + \|v_s\|_{H^1}^2\le 2\kappa^2$ for any $s\le \tau_\kappa$. 

To estimate $I^2_t$ we observe that by Assumption \ref{t:ws} there is $C>0$ for which $\big|\sigma(a)^{-1} \sigma'(a) - \sigma(b)^{-1} \sigma'(b)\big|  \le C |a-b|$, so that, from Cauchy-Schwartz inequality and arguing as before,
\[
\begin{split}
|I^2_t|  & \le C \| j \|_{L^2}^2 \int_0^{t\wedge \tau_\kappa}\!  \int\! f_s(x) |u_s(x)-v_s(x)| \, \rmd x\,\rmd s \le C  \int_0^{t\wedge \tau_\kappa}\!  \|f_s\|_{H^1} \|u_s-v_s\|_{L^2} \\ & \le C \int_0^{t\wedge \tau_\kappa}\! \sqrt{\Psi_s}\, \|u_s - v_s\|_{L^2} \,\rmd s \;.
\end{split}
\]

In view of the definition of $h(u)$ and Assumption \ref{t:ws}, it is easy to show that, for a suitable $C>0$,  
\be
\label{abh}
\max\{(a-b)^2;(h(a)-h(b))^2\} \le C (h(a)-h(b))(a-b)\;.
\ee
Therefore, by the previous estimates and Young inequality, 
\be
\label{r12}
|I^1_t| + |I^2_t| \le \frac 12 \int_0^{t\wedge \tau_\kappa}\! \big\langle  h(u_s)-h(v_s), u_s - v_s\big\rangle_{L^2}\,\rmd s  + C(1+\kappa^2) \int_0^{t\wedge \tau_\kappa}\! \Psi_s \,\rmd s \;.
\ee

To estimate $I^3_t$, we first notice that there is $C>0$ such that, for any $f\in H^1$ and $g\in H^{-1/2}$, 
\be
\label{fg}
\|f g\|_{H^{-1}} \le C \|f\|_{H^1} \|g\|_{H^{-1/2}}\;.
\ee
Indeed, by Fourier expansion and Parseval identity,
\[
\begin{split}
\|fg\|_{H^{-1}}^2 & = \sum_{k\in\bb Z^d} \frac{\big|\hat{fg}(k)\big|^2}{1+|k|^2}  = \sum_{k\in\bb Z^d} \frac1{1+|k|^2} \bigg|\frac1{(2\pi)^{d/2}}\sum_{k'\in \bb Z^d} \hat f(k-k') \hat g (k') \bigg|^2 \\ & \le \sum_{k\in\bb Z^d} \frac1{1+|k|^2} \bigg(\sum_{k'\in \bb Z^d} |\hat f(k-k')| |\hat g (h)| \bigg)^2 \\ & \le \sum_{k\in\bb Z^d} \frac1{1+|k|^2} \sum_{k'\in \bb Z^d} |\hat f(k-k')|^2  (1+|k-k'|^2)\sum_{k'\in \bb Z^d} \frac{|\hat g(k')|^2}{1+|k-k'|^2} \\ & = \|f\|_{H^1}^2 \sum_{k'\in \bb Z^d} |\hat g(k')|^2 \sum_{k\in\bb Z^d} \frac1{(1+|k|^2)(1+|k-k'|^2)} \le C \|f\|_{H^1}^2 \|g\|_{H^{-1/2}}^2 \;,
\end{split}
\]
where we used that there is $C>0$ such that, for $d=1,2,3$,
\[
\begin{split}
& \sum_{k\in\bb Z^d} \frac1{(1+|k|^2)(1+|k-k'|^2)} \le \frac C{\sqrt{1+|k'|^2}} \;.
\end{split}
\]
By \eqref{fg} and standard interpolation,
\be
\label{r31}
\begin{split}
& \sum_{k\in\bb Z^d} \big\langle \beta_s j * e_k , R_1 \beta_s j * e_k \big\rangle_{L^2} = \sum_{k\in\bb Z^d} \|\beta_s j* e_k\|_{H^{-1}}^2 \\ & \le C\sum_{k\in\bb Z^d} \|j*e_k\|_{H^1}^2 \|\beta_s\|_{H^{-1/2}}^2 \le C \| j\|_{H^1}^2 \|\beta_s\|_{H^{-1/2}}^2 \le C \|\beta_s\|_{L^2} \|\beta_s\|_{H^{-1}}\;.
\end{split}
\ee

By the definition of $h(\cdot)$ and Assumption \ref{t:ws} it is straightforward to verify that $|\beta_s|\le C|h(u_s)-h(v_s)|$. Moreover, we claim that $\gamma_s := (h(u_s)-h(v_s))^{-1}\beta_s \in L^\infty\cap H^1$ and that, for a suitable $C>0$, $\|\gamma_s\|_{H^1} \le C( 1+\|u_s\|_{H^1}+\|v_s\|_{H^1})$. To see this, notice that the function $\tilde\sigma(r) := \sigma(h^{-1}(r))^{-1/2}$ is $C^2$ with bounded derivatives, and satisfies
\[
\gamma_s = \int_0^1\! \tilde\sigma'(h(v_s(x))+\lambda(h(u_s(x))-h(v_s(x))))\,\rmd \lambda\;,
\]
from which the claim follows. 

By using \eqref{fg}, for any $s\le \tau_\kappa$, 
\[
\big\|\beta_s\big\|_{H^{-1}} \le C \big\|\gamma_s\big\|_{H^1} \big\|h(u_s)-h(v_s)\big\|_{H^{-1/2}} \le C (1+ 2\kappa) \big\|h(u_s)-h(v_s)\big\|_{H^{-1/2}}\;,
\] 
hence, by \eqref{r31}, interpolation and Young inequality, for any $\eps>0$,
\[
\begin{split}
|I^3_t| & \le C (1+ 2\kappa)  \int_0^{t\wedge \tau_\kappa}\! \big\|h(u_s)-h(v_s)\big\|_{L^2} \big\|h(u_s)-h(v_s)\big\|_{H^{-1/2}}\,\rmd s \\ & \le C (1+ 2\kappa)  \int_0^{t\wedge \tau_\kappa}\! \big\|h(u_s)-h(v_s)\big\|_{L^2}^{3/2} \big\|h(u_s)-h(v_s)\big\|_{H^{-1}}^{1/2}\,\rmd s \\ & \le \eps \int_0^{t\wedge \tau_\kappa}\! \big\|h(u_s)-h(v_s)\big\|_{L^2}^2\,\rmd s + C_\eps  (1+\kappa^2)\int_0^{t\wedge \tau_\kappa}\!  \big\|h(u_s)-h(v_s)\big\|_{H^{-1}}^2 \,\rmd s
\;.
\end{split}
\]
Since by \eqref{abh} $\big\|h(u_s)-h(v_s)\big\|_{L^2}^2 \le \big\langle  h(u_s)-h(v_s), u_s - v_s\big\rangle_{L^2}$, by choosing $\eps$ small enough we conclude that there is $C>0$ such that,
\be
\label{r3}
|I^3_t|  \le \frac 12 \int_0^{t\wedge \tau_\kappa}\! \big\langle  h(u_s)-h(v_s), u_s - v_s\big\rangle_{L^2}\,\rmd s  + C (1+\kappa^2) \int_0^{t\wedge \tau_\kappa}\! \Psi_s \,\rmd s \;.
\ee

By \eqref{dt}, \eqref{r12}, and \eqref{r3} we get,
\[
\Psi_t \le C \int_0^{t\wedge \tau_\kappa}\! \Psi_s \,\rmd s + M_{t\wedge\tau_\kappa} \le C \int_0^t\! \Psi_s \,\rmd s + M_{t\wedge\tau_\kappa} \;.
\]
By the optional stopping theorem $\mc E (M_{t\wedge\tau_\kappa}) = \mc E(M_0) =0$; therefore, by taking the expectation in both sides and applying Gronwall's inequality we conclude that $\mc E(\Psi_t) = 0$, and therefore $\mc P$-a.s.\ $\Psi_t=0$.
\end{proof}

By the Yamada-Watanabe argument \cite{YW} (see also \cite{RY,RSZ}) we next deduce uniqueness of the law of weak solutions.

\begin{proposition}
\label{prop:5.2}
Given $\bar u_0\in H^1$, the following holds.
\begin{itemize}
\item[a)] The law $\bb Q = \mc P \circ (u,\alpha)^{-1}$ on $C(L^2) \times C(H^{-\bar{s}})$ is the same for any weak solution $((\Omega,\mc G,\mc G_t,\mc P),(u,\alpha))$ to \eqref{1} with initial datum $\bar u_0$.
\item[b)] There exists a Borel map $\Theta \colon C(H^{-\bar{s}}) \to C(L^2)$, $\mc B_t(C(H^{-\bar{s}}))/\mc B_t$ measurable, and such that for any weak solution $((\Omega,\mc G,\mc G_t,\mc P),(u,\alpha))$ we have $u=\Theta\circ\alpha$ $\mc P$-a.s..
\end{itemize}
\end{proposition}

\begin{proof}
The proof can be easily achieved by adapting the argument in \cite{RY} for finite dimensional diffusions. However, for the reader's convenience, we present the complete strategy. 

\medskip
\noindent
a) Fix $\bar{u}_0 \in H^1$ and let $((\Omega^i,\mc G^i,\mc G^i_t,\mc P^i),(u^i,\alpha^i))$, $i=1,2$, be two weak solutions to \eqref{1} with initial condition $\bar u_0$. To take advantage of the pathwise uniqueness proved in Proposition \ref{prop:5.1}, we need to bring them on a same filtered probability space. 

Denote by $\bb P^*$ the law of the cylindrical Wiener process on $C(H^{-\bar{s}})$. Set also $\bb Q^i= \mc P^i \circ (u^i,\alpha^i)^{-1} $ be the probabilities on $C(L^2)\times C(H^{-\bar{s}})$ induced by the pair $(u^i,\alpha^i)$. As the spaces involved are Polish, these probabilities can be disintegrated w.r.t.\ $\bb P^*$ so that
\[
\bb Q^i ( \rmd v, \rmd v')=   \bb Q^i_{v'} (\rmd v) \, \bb P^*(\rmd v') \;, \quad i=1,2\;,
\]
where $\bb P^*$-a.s.\ $ \bb Q^i_{v'}$ is a probability on $C(L^2)$. Moreover, $v' \mapsto \bb Q^i_{v'}(A)$ is a Borel map for any $A \in \mc B(C(L^2))$. In the sequel, we need the following result, which is a straightforward adaptation to the present context of \cite[Chap.\ 4, Lemma (1.6)]{RY}.

\begin{lemma}
\label{lem:5.3}
If $A\in \mc B_t(C(L^2))$, $t\in [0,T]$, the map $w^3 \mapsto \bb Q^i_{w^3}(A)$ is $\mc B_t(C(H^{-\bar{s}}))$-measurable up to a negligible set.
\end{lemma}

Consider now the product space $\bs W := C(L^2) \times C(L^2) \times  C(H^{-\bar{s}}) $, whose elements are denoted by $w = (w^1,w^2,w^3)$. On $\bs W$ we define the probability measure,
\[
\Pi(\rmd w^1,\rmd w^2,\rmd w^3) := \bb Q^1_{w^3} (\rmd w^1) \bb Q^2_{w^3} (\rmd w^2) \bb P^*(\rmd w^3) \;,
\]
and we endow $\bs W$ with the filtration $\mc G_t$ defined as the completion with respect to $\Pi$ of the canonical filtration $\mc B_t(\bs W)$.

We now claim that, for $i=1,2$, $((\bs W,\mc G,\mc G_t,\Pi),(w^i,w^3))$ are weak solutions to \eqref{1} with initial condition $\bar u_0$ on the same filtered space, and such that $\bb P^i=\mc P^i \circ (u^i)^{-1}$ is the law of $w^i$. The latter assertion, which is immediate by construction, clearly implies condition ii) in the definition of weak solution, hence it remains to verify conditions i) and iii).

To prove that $w^3$ is a $L^2$-cylindrical Wiener process with respect to $\mc G_t$, we only need to check that for any $0\le s<t\in [0,T]$ the process $w^3_t-w^3_s$ is independent of $\mc G_s$. This property follows by noticing that, letting $A_1,A_2\in \mc B_s(C(L^2))$, $B\in\mc B_s(C(H^{-\bar{s}}))$, by Lemma \ref{lem:5.3}, for any $\psi\in C^\infty(\bb T^d)$,
\[
\begin{split}
& \bb E^{\Pi}\big[ \exp\big(\rmi \langle \psi, w^3_t - w^3_s\rangle \big) \id_{w^1\in A_1}\id_{w^2\in A_2} \id_{w^3\in B}\big] \\ & \qquad = \int_B\! \exp\big(\rmi \langle \psi, w^3_t - w^3_s\rangle \big) \bb Q^1_{w^3} (A_1) \bb Q^2_{w^3} (A_2) \bb P^*(\rmd w^3) \\ & \qquad = \int_B\! \bigg[\int\!\exp\big(\rmi \langle \psi, \bar w_t - \bar w_s\rangle \big)\,\bb P^*_{t,w^3}(\rmd \bar w)\bigg] \bb Q^1_{w^3} (A_1) \bb Q^2_{w^3} (A_2) \bb P^*(\rmd w^3) \\ & \qquad = \exp\big(- (t-s) \|\psi\|_{L^2}^2 \big) \int_B\! \bb Q^1_{w^3} (A_1) \bb Q^2_{w^3} (A_2) \bb P^*(\rmd w^3) \\ & \qquad = \exp\big(- (t-s) \|\psi\|_{L^2}^2 \big)\,  \Pi(A_1\times A_2\times B)\;,
\end{split}
\]
where $\bb P^*_{t,w^3}$ is a regular version of the conditional probability $\bb P^*(\,\cdot\,|w^3_s,s\in [0,t])$. 

Let $\psi\in C^\infty([0,T]\times \bb T^d)$ and $t\in [0,T]$. Since $\mc P^i$-a.s.,
\[
\begin{split}
\langle u^i_t, \psi_t \rangle_{L^2} & = \langle \bar u_0, \psi_0\rangle_{L^2}  + \int_0^t\! \langle u^i_s, \partial_s \psi_s\rangle_{L^2}\,\rmd s \\ & \quad + \int_0^t\! \langle \sigma(u^i_s) (\Delta u^i_s - W'(u^i_s)), \psi_s\rangle_{L^2} \, \rmd s + \int_0^t\! \langle \psi_s, B(u^i_s)\,\rmd\alpha^i_s \rangle_{L^2}\;,
\end{split}
\]
then, as follows from, e.g., \cite[Chap.\ 4, Ex.\ (5.16)]{RY}, $\Pi\,$-a.s.,
\[
\begin{split}
\langle w^i_t, \psi_t \rangle_{L^2} & = \langle \bar u_0, \psi_0\rangle_{L^2}  + \int_0^t\! \langle w^i_s, \partial_s \psi_s\rangle_{L^2} \,\rmd s\\ & \quad + \int_0^t\! \langle \sigma(w^i_s) (\Delta w^i_s - W'(w^i_s)), \psi_s\rangle_{L^2} \, \rmd s + \int_0^t\! \langle \psi_s, B(w^i_s)\, \rmd w^3_s \rangle_{L^2}\;,
\end{split}
\]
which is property iii) in the definition of weak solutions.

By the pathwise uniqueness in Proposition \ref{prop:5.1}, $\Pi((w^1_t,w^3_t) = (w^2_t,w^3_t) \; \forall\, t\in [0,T]) = 1$, which implies $\bb Q^1=\bb Q^2$, i.e., the uniqueness of the law of weak solutions.  

\medskip
\noindent
b) Let $((\Omega,\mc G,\mc G_t,\mc P),(u,\alpha))$ be a weak solution of \eqref{1} with initial condition $\bar u_0$, whose existence is ensured by Lemma \ref{lem:5.1}. We apply the previous construction with $((\Omega^i,\mc G^i,\mc G^i_t,\mc P^i),(u^i,\alpha^i)) = ((\Omega,\mc G,\mc G_t,\mc P),(u,\alpha))$ for $i=1,2$. Thus, for 
\[
\mc P \circ (u,\alpha)^{-1} = \bb Q\;, \qquad \bb Q(\rmd w^i,\rmd w^3) = \bb Q_{w^3} (\rmd w^i) \, \bb P^*(\rmd w^3)\;, \quad i = 1,2\;, 
\]
\[
\Pi(\rmd w^1,\rmd w^2,\rmd w^3) = \bb Q_{w^3} (\rmd w^1) \bb Q_{w^3} (\rmd w^2) \bb P^*(\rmd w^3) \;,
\] 
pathwise uniqueness yields $\Pi(w^1_t = w^2_t \; \forall\, t\in [0,T]) = 1$. As a consequence, the processes $w^1$ and $w^2$ are simultaneously equal and independent under the measure $\bb Q_{w^3} (\rmd w^1) \bb Q_{w^3} (\rmd w^2)$ for $\bb P^*$-a.s\ $w^3$. This is possible only if there exists a Borel map $\Theta\colon C(H^{-\bar{s}}) \to C(L^2)$ such that $\bb Q_{w^3} = \delta_{\Theta(w^3)}$ for $\bb P^*$-a.s\ $w^3$. Furthermore, in view of Lemma \ref{lem:5.3}, $\Theta$ is $\mc B_t(C(H^{-\bar{s}}))/\mc B_t(C(L^2))$ measurable. Therefore, $\bb Q(\rmd w^1,\rmd w^3) = \delta_{\Theta(w^3)} (\rmd w^1) \,\bb P^*(\rmd w^3)$, whence $u=\Theta\circ\alpha$ $\mc P$-a.s.. 
\end{proof}

\begin{proof}[Proof of Theorem \ref{t:eu}] 
By Theorem \ref{th:4.1}, for each initial datum $\bar u_0\in H^1$ there exists a martingale solution $\bb P$ to \eqref{1} satisfying \eqref{a4terdue}, which implies \eqref{steu} in view of the growth assumptions of the potential $W$. Moreover, $\bb P(u\in C(H^1)) =1$ in view of Lemma \ref{lem:5.2}. By Lemma \ref{lem:5.1} and item a) of Proposition \ref{prop:5.2}, the uniqueness of the martingale solution $\bb P$ follows. 

The pathwise uniqueness proved in Proposition \ref{prop:5.1} clearly implies the uniqueness of strong solutions. Therefore, we are left with the proof of existence of strong solutions.

Given $\bar u_0\in H^1$ and a probability space $(\Omega,\mc G,\mc P)$ equipped with a cylindrical Wiener process $\alpha$, we claim that the process $u=\Theta\circ \alpha$, with $\Theta$ as given in item b) of Proposition \ref{prop:5.2}, is a strong solution to \eqref{1} with initial datum $\bar u_0$. To show this, we observe that the law of $(u,\alpha)$ is equal to the law $\bb Q$ of weak solutions, uniquely determined according to Proposition \ref{prop:5.2}. In addition, as seen in the proof of that proposition, denoting by $(w^1,w^2)$ the elements of $C(L^2)\times C(H^{-\bar{s}})$ and by $\tilde {\mc G}$ [resp.\ $\tilde{\mc G}_t$] the $\sigma$-algebra $\mc B(C(L^2)\times C(H^{-\bar{s}}))$  [resp.\ filtration $\mc B_t(C(L^2)\times C(H^{-\bar{s}}))$] completed under $\bb Q$, the pair $((C(L^2)\times C(H^{-\bar{s}}),\tilde{\mc G},\tilde{\mc G}_t,\bb Q), (w^1,w^2))$ is a weak solution to \eqref{1} with initial datum $\bar u_0$. Therefore, setting $\mc G_t = (u,\alpha)^{-1}(\tilde{\mc G}_t)$, the same reasoning as in the proof of Proposition \ref{prop:5.2}, based on \cite[Chap.\ 4, Ex.\ (5.16)]{RY}, implies that the pair $((\Omega,\mc G_t,\mc G,\mc P),(u,\alpha))$ is a weak solution to \eqref{1} with initial datum $\bar u_0$. Since $\Theta$ is $\mc B_t(C(H^{-\bar{s}}))/\mc B_t(C(L^2))$ measurable, the process $u$ is $\mc G^\alpha_t$-adapted, hence $u$ is a strong solution. 
\end{proof} 
 
\appendix

\section{A class of $C_0$-semigroups}
\label{app:a}

In this appendix we prove the following lemma, concerning the generation of $C_0$-semigroups on the Sobolev space $H^1$. Here we set $\mc H:=H^1$ and denote the norm and inner product in $\mc H$ simply by $\|\cdot\|$ and $\langle \cdot,\cdot\rangle$. 

\begin{lemma}
\label{semigroups}
Let $\mc H=H^1$, $d=2,3$, $v \in H^2$ and $A\colon H^3\subset \mc H \to \mc H$ defined by $Au = \sigma(v) \Delta u$.  Then the following holds.
\begin{enumerate}
\item $A$ is closed, densely defined, and it generates a $C_0$-semigroup $S(t)$, $t\ge 0$, on $\mc H$ satisfying $\| S(t) \| \le e^{m_0 t}$ for any $t\ge 0$ for some $m_0>0$ (depending only on $\sigma$).
\item Given $\delta>0$ let $R_\delta := (\mathrm{Id} -\delta \Delta)^{-1}$ and $A_\delta \colon \mc H \to \mc H$ be defined by $A_\delta u := \sigma(R_\delta v) R_\delta \Delta R_\delta u$. Then $A_\delta$ is a bounded (indeed compact) operator on $\mc H$ and generates a uniformly continuous (semi)group of linear operators $S_\delta(t)$, $t\ge 0$. Moreover, $\| S_\delta(t) \| \le e^{m_0 t}$ for any $\delta>0$ small enough (depending only on $\sigma(v)$) and any $t\ge 0$.   
\item Consider the linear operator $\lim_{\delta \to 0} A_\delta$ defined on $\{ u \in \mc H \colon \exists \lim_{\delta \to 0} A_\delta u \}$ as $(\lim_{\delta \to 0} A_\delta) u:= \lim_{\delta \to 0} A_\delta u$.  Then $\lim_{\delta \to 0} A_\delta =A$ as unbounded operators and $S_\delta(t) u \to S(t)u$ in $\mc H$ for every $u \in \mc H$ and every $t\ge 0$.
\end{enumerate} 
\end{lemma}

\begin{proof}
(1) The operator $A$ is densely defined and closed. Indeed, let $Au_n=f_n \to f$ in $\mc H$ and $\{ u_n \} \subset H^3$, $u_n \to u$ in $\mc H$. Since $v \in H^2$ and $d\le 3$, it is easy to check that multiplication by $\sigma(v)^{-1}$ is a bounded operator on $\mc H$, hence $\Delta u_n =\sigma(v)^{-1} f_n \to \sigma(v)^{-1} f$ in $\mc H$. By elliptic regularity $\{ u_n\} \subset H^3$ is bounded, hence $u \in H^3$. In addition, $u_n \to u$ in $H^2$ and therefore $A u=f$ in $L^2$ and in turn in $\mc H$.

The rest of statement (1) follows from the Lumer-Phillips theorem \cite{Pazy} once we prove that there exists $m_0>0$ such that

a)  $\| ((m+m_0)\mathrm{Id} -A  ) u \| \ge m \| u \|$ for any $m>0$ and $u\in H^3$;

b) $A-(m+m_0) \mathrm{Id}$ is surjective for each $m>0$.

To prove a) it is clearly enough to show that $\langle( m_0 \mathrm{Id} -A) u, u\rangle \ge 0$ for some $m_0>0$ and any $u \in H^3$. We have,
\[
\begin{split}
\langle( m_0 \mathrm{Id} -A) u, u\rangle & = m_0 \|u\|^2 - \int\! \sigma(v) u \Delta u \,\rmd x + \int\! \sigma(v) \left(  \Delta u \right)^2 \, \rmd x \\ & \ge m_0 \| u \|^2 + (\inf \sigma) \int\! (\Delta u)^2 \,\rmd x - (\sup \sigma) \int |u| | \Delta u| \, \rmd x \\ & \ge \left(m_0-  \frac{(\sup \sigma)^2}{2(\inf \sigma) }\right)  \| u\|^2+ \frac12 (\inf \sigma)  \int\! (\Delta u)^2 \, \rmd x\;, 
\end{split} 
\]
where we used Young's inequality in the last step. Claim a) follows for $m_0$ large enough.

To prove b), for $\lambda \in [0,1]$ we consider the family of bounded operators $A^\lambda \colon H^3 \to \mc H$ defined by $A^\lambda u =-(m+m_0) u +\left(\lambda \sigma(v)+1-\lambda \right) \Delta u$. Notice that $[0,1] \ni \lambda \to A^\lambda \in \mc L (H^3; \mc H)$ is norm continuous and $A^0= -(m+m_0) \mathrm{Id} +\Delta$ is  surjective (actually a Banach space isomorphism). Thus, by the continuity method, claim b) follows once we prove that there exists  $c>0$ such that $\| u\|_{H^3}  \le c \| A^\lambda u\|$ for any $u\in H^3$ and $\lambda \in [0,1]$.
Arguing as in the last displayed formula we get,
\[ 
\left(m+m_0 - \frac{( \max \{1,\sup \sigma \})^2}{2 \min\{( 1, \inf \sigma) \}} \right)\| u\|^2+ \frac12  \min \{1, \inf \sigma \}  \int\! (\Delta u)^2 \, \rmd x   \le  \| A^\lambda u\| \|u\|\;, 
\]
and, by Young's inequality, $\| u\| \le c \| A^\lambda u\|$ for a suitable $c>0$ depending only on $m_0, m$ and $\sigma$. 

Finally, notice that $\Delta u =\left(\lambda \sigma(v)+1-\lambda \right)^{-1} \left( A^\lambda u   + (m+m_0) u \right) \in \mc H$, which together with the previous inequality gives $\| \Delta u\| \le C(m_0,m,\sigma,v) \| A^\lambda u + (m+m_0) u \| \le C^\prime(m_0,m,\sigma,v) \| A^\lambda u\|$ and the conclusion follows by elliptic regularity.  

\smallskip
(2) First we notice that $R_\delta \colon H^s \to H^{s+2}$ is continuous (a linear isomorphism), $\Delta R_\delta u =R_\delta \Delta u$, $\| R_\delta u \|_{H^s} \le \| u\|_{H^s}$, and $R_\delta u \to u$ in $H^s$ as $\delta \to 0$ for any $u\in H^s$ and $s \in \bb R$. Note also that $R_\delta v \in H^4 \subset C^2$ for $d=2,3$, hence $\sigma(R_\delta v) \in C^2$. Since $R_\delta \Delta R_\delta \colon \mc H \to H^3$ is bounded, $H^3 \hookrightarrow H^2$ is compact, and multiplication by $\sigma(R_\delta v)$ is  bounded on $H^2$ we see that $A_\delta \colon \mc H \to H^2$ is compact, hence $A_\delta \colon \mc H \to \mc H$ is a compact operator. Therefore, $A_\delta$ generates a uniformly continuous (semi)group of linear operators $S_\delta(t)$, $t\ge 0$. As in part (1) above, in view of the Lumer-Phillips theorem, the exponential estimate follows once we prove that there exists $m_0>0$ such that, for any $\delta>0$ small enough,

a)  $\| ((m+m_0)\mathrm{Id} -A_\delta  ) u \| \ge m \| u \|$ for any $m>0$ and $u\in H^3$;

b) $A_\delta-(m+m_0) \mathrm{Id}$ is surjective for each $m>0$.

To prove a), again it is enough to check that there exists $m_0>0$ such that for any $\delta>0$ small enough $\langle( m_0 \mathrm{Id} -A_\delta) u, u\rangle \ge 0$ for any $u \in H^3$. Integrating by parts,
\[
\langle( m_0 \mathrm{Id} -A_\delta) u, u\rangle = m_0 \| u\|^2 + I_\delta + I\!I_\delta\;,
\]
where
\[
I_\delta: = - \int\! \sigma(R_\delta v) u R_\delta \Delta R_\delta u  \, \rmd x \;, \qquad I\!I_\delta := \int\! \Delta u \, \sigma(R_\delta v) R_\delta \Delta R_\delta u \, \rmd x \:.
\]
Now, using Sobolev inequality,
\[
\begin{split}
I_\delta & = \int\! \left( \sigma(R_\delta v) \nabla u \nabla R^2_\delta u + u \sigma^\prime(R_\delta v) \nabla R_\delta v \nabla R^2_\delta u  \right) \, \rmd x \\ & \ge - \left( \sup \sigma + C \| \sigma' \|_\infty  \| v\|_{H^2} \right) \| u\|^2 \;,  
\end{split}
\]
which leads to chose $m_0=\sup \sigma + C \| \sigma' \|_\infty  \| v\|_{H^2}$. On the other hand, as $\mathrm{Id}=(\mathrm{Id}-\delta \Delta) R_\delta$ we have,
\[
\begin{split}
I\!I_\delta & = \int\! \sigma(R_\delta v) |\Delta R_\delta u|^2 \, \rmd x + \int\! \Delta R_\delta u  \big( \mathrm{Id} -\delta \Delta \big) \left[ \sigma(R_\delta v),  R_\delta \right]  \Delta R_\delta u  \, \rmd x \\ & \ge \Big( \inf \sigma - \left\| \big(\mathrm{Id}-\delta \Delta \big)  \left[\sigma(R_\delta v),  R_\delta \right] \right\|_{L^2_0 \to L^2}  \Big)\int\! |\Delta R_\delta u|^2 \, \rmd x \; ,
\end{split}
\]
since $\Delta R_\delta u \in L^2_0$, the closed subspace of functions with zero average, and claim a) follows once we show that $\left\| \big(\mathrm{Id}-\delta \Delta \big)  \left[ \sigma(R_\delta v),  R_\delta \right] \right\|_{L^2_0 \to L^2} = o(1) $ as $\delta \to 0$.

To estimate the commutator, notice that for $w \in L^2_0$, $g=R_\delta w$ and $f=R_\delta ( \sigma(R_\delta v) w)$, we have $f, g \in H^2$, $g -\delta \Delta g= w$, $f- \delta \Delta f= \sigma( R_\delta v) w$ and 
\[
\begin{split}
\big(\mathrm{Id}-\delta \Delta \big) & \left[\sigma(R_\delta v),  R_\delta \right] w = \big(\mathrm{Id}-\delta \Delta \big) (-f+\sigma(R_\delta v) g ) \\ & = - \delta \big(  g \Delta \sigma(R_\delta v) +2 \nabla g \nabla \sigma(R_\delta v) \big)  
\\ & = - \delta \big[ \sigma^{\prime} (R_\delta v)  g  R_\delta  \Delta v + \sigma^{\prime \prime}(R_\delta v ) g |\nabla R_\delta v|^2 +2 \sigma^{\prime} (R_\delta v) \nabla g \nabla R_\delta v \big] \; .
\end{split}
\]
If $\varphi \in L^2_0$ and $\psi=R_\delta \varphi \in H^2$ solves $\psi -\delta \Delta \psi= \varphi$, then $\int \psi^2 +\delta |\nabla \psi|^2 \le \| \varphi\|_{L^2} \| \psi\|_{L^2} $, whence $\| \nabla \psi\|_{L^2} \le \delta^{-1/2} \| \varphi\|_{L^2}$ and, by H\"older and Sobolev embedding $H^1 \hookrightarrow L^6$, we have $\| \psi\|_{L^p} \le C\delta^{(6-3p)/4p} \| \varphi\|_{L^2}$ for any $2\le p\le 6$. Analogously, $\int  \delta |\nabla \psi|^2 +\delta^2 |\Delta \psi|^2 \le  \delta \| \varphi\|_{L^2} \| \Delta \psi\|_{L^2} $, whence $\| \Delta \psi\|_{L^2} \le \delta^{-1} \| \varphi\|_{L^2}$ and by Holder, Sobolev  and Calderon-Zygmund inequalities we have $\| \nabla \psi\|_{L^p} \le C\delta^{(6-5p)/4p} \| \varphi\|_{L^2}$ for any $2\le p\le 6$. Applying these estimates to $\varphi=w$ and $\varphi=\Delta  v \in L^2_0$ we have,
\[ 
\left\| \big(\mathrm{Id}-\delta \Delta \big)  \left[  \sigma(R_\delta v),  R_\delta \right] w \right\|_{L^2} \le   \delta \| \sigma\|_{C^2} \left( \| g\|_{L^4} \|R_\delta \Delta v \|_{L^4} +\| g\|_{L^6} \| \nabla R_\delta v \|_{L^6}^2 \right. \]
\[ \left. + \| \nabla g\|_{L^4} \| \nabla R_\delta v\|_{L^4} \right) \le C \delta^{1/8} \| \sigma\|_{C^2} (1+ \| v\|_{H^2}^2) \| w\|_{L^2}\;,
\]
so that $\left\| \big(\mathrm{Id}-\delta \Delta \big)  \left[  \sigma(R_\delta v),  R_\delta \right]  \right\|_{L^2_0 \to L^2 } \leq C\delta^{1/8}$ and the claim follows.
 
To prove b), notice that $A_\delta$ is a compact operator, hence $A_\delta-(m+m_0) \mathrm{Id}$ is Fredholm of index zero. Since it is injective by part a), then it is surjective. 

(3) Concerning the first statement, we notice that $\lim_{\delta \to 0} A_\delta u=Au$ for any $u \in H^3$ because $R_\delta \Delta R_\delta u \to \Delta u$ in $H^1$, $\sigma(R_\delta v) \to \sigma(v)$ in $H^2$ and the product is jointly continuous for $d=2,3$. Conversely, suppose $\lim_{\delta \to 0} A_\delta u$ exists for some $u \in \mc H$, we claim that $u \in H^3$ and the limit is $Au$. To see this, notice that $A_\delta u$ is bounded in $\mc H$, hence $R_\delta \Delta R_\delta u=\sigma(R_\delta v)^{-1} A_\delta u$ is also bounded in $\mc H$, which implies $\Delta u \in \mc H$, where the Laplacian is taken in the sense of distributions. Then $u \in H^3$ by elliptic regularity and the conclusion follows from the initial observation.

To finish the proof it is enough to apply \cite[Thm.\ 5.2]{T} to infer convergence of semigroups from convergence of the corresponding resolvent operators at some common point. Fix $m_0>0$ as in part (1) and (2) above and $m>0$ so that both $(m+m_0) \mathrm{Id}-A$ and $(m+m_0) \mathrm{Id}-A_\delta$ are injective on their respective domains and onto. We claim that $ u_\delta:= \left((m+m_0) \mathrm{Id}-A_\delta \right)^{-1}f  \to  \left((m+m_0) \mathrm{Id}-A\right)^{-1}f=:u$ as $\delta \to 0$ for any $f \in \mc H$. 
Indeed, by definition $(m+m_0) u_\delta- A_\delta u_\delta =f$ and in view of the dissipativity inequality (part (2)  of the proof, claim a)), we have $m \| u_\delta\| \le \| f \|$ but indeed even $\| \Delta R_\delta u_\delta \|_{L^2} \le C(m,m_0, \| v\|_{H^2})\| f \|$ for $\delta>0$ small enough. Thus, by elliptic regularity $R_\delta u_\delta $ is bounded in $H^2$ and, up to subsequences, $R_\delta u_\delta \to \bar{u}$  strongly in $H^1$ and weakly in $H^2$ as $\delta \to 0$ for some $\bar u \in H^2$ possibly depending on the subsequence. Observe that $u_\delta-R_\delta u_\delta=\delta \Delta R_\delta u_\delta \to 0$ in $L^2$, hence $u_\delta \rightharpoonup \bar{u}$ weakly in $H^1$. Since $\Delta R_\delta R_\delta u_\delta= \sigma(R_\delta v)^{-1}  \left( (m+m_0)u_\delta -f \right)$ is also bounded in $\mc H$, hence $R_\delta R_\delta u_\delta$ is bounded in $H^3$ by elliptic regularity, $R_\delta R_\delta u_\delta \to \bar{u}$ in $H^1$, which in turn gives $\bar{u} \in H^3$ and $R_\delta \Delta R_\delta u_\delta \rightharpoonup \Delta u$ weakly in $H^1$.

Since $H^2 \hookrightarrow C^0$ for $d=2,3$ and $\sigma(R_\delta v) \to \sigma(v)$ uniformly, taking $L^2$ scalar product of the equation for $u_\delta$ with some $g\in L^2$, as $\delta \to 0$ we get $\langle (m+m_0) \bar{u}- A \bar{u}, g\rangle_{L^2} =\langle f,g\rangle_{L^2}$, which in turn gives $(m+m_0) \bar{u}- A \bar{u} =f$ because $g$ is arbitrary. By injectivity, $\bar{u}=u$ is independent of the chosen subsequence and so far we get $u_\delta \rightharpoonup u$ weakly in $\mc H$ and the proof is complete once we show that $\| u_\delta \| \to \| u \|$ as $\delta \to 0$.
 In order to conclude, we argue as in part (2) above and and we write,
\[
\begin{split}
& (m_0+m)\| u_\delta\|^2 + \int\! \sigma(R_\delta v) |\Delta R_\delta u_\delta|^2 \, \rmd x = \langle f,u_\delta \rangle \\ & \qquad + \int\! \sigma(R_\delta v)u_\delta R_\delta \Delta R_\delta u \, \rmd x - \int\! \Delta R_\delta u_\delta \big( \mathrm{Id} -\delta \Delta \big)  \left[ \sigma(R_\delta v),  R_\delta \right]  \Delta R_\delta u_\delta\, \rmd x \;.
\end{split}
\]
Since $\left\| \big(\mathrm{Id}-\delta \Delta \big)  \left[  \sigma(R_\delta v),  R_\delta \right] \right\|_{L^2_0 \to L^2} =o(1) $ as $\delta \to 0$, the previous convergence properties yields,
\[
\begin{split}
& \varlimsup_{\delta \to 0} \Big[ (m_0+m)\| u_\delta\|^2 +\int \sigma(R_\delta v) |\Delta R_\delta u_\delta|^2 \, \rmd x \Big] = \langle f,u_\delta \rangle + \int\! \sigma(v)u  \Delta  u \, \rmd x \\ & \qquad = (m_0+m)\| u \|^2 + \int\! \sigma(v) |\Delta u|^2 \, \rmd x\;.
\end{split}
\]
On the other hand, by $L^2$-weak lower semicontinuity,
\[
\begin{split}
& \varlimsup_{\delta \to 0}  (m_0+m)\| u_\delta\|^2 +\int\! \sigma( v) |\Delta u|^2 \, \rmd x \le \varlimsup_{\delta \to 0}  (m_0+m)\| u_\delta\|^2 \\ & \quad +\varliminf_{\delta\to 0} \int\! \sigma(R_\delta v) |\Delta R_\delta u_\delta|^2 \, \rmd x \le \varlimsup_{\delta \to 0} \big[ (m_0+m)\| u_\delta\|^2 + \int\! \sigma(R_\delta v) |\Delta R_\delta u_\delta|^2 \, \rmd x \big]\;,
\end{split}
\]
hence $\varlimsup_{\delta \to 0}\| u_\delta\|^2 \le \| u\|^2$ and therefore $\| u_\delta\| \to \| u\|$ as $\delta \to 0$ as claimed.
\end{proof}

\medskip

\subsection*{Acknowledgements}
We thank M.\ Mariani for useful discussions.


\begin{thebibliography}{99}

\bibitem{AL}  Alt, H.W., Luckhaus, S.: \textit{Quasilinear elliptic-parabolic differential equations.} Math. Z. \textbf{183}, 311--341 (1983). 

\bibitem{AR} Albeverio, S., R\"ockner, M.,: \textit{Stochastic differential equations in infinite dimensions: solutions via Dirichlet forms.} Probab. Theory Related Fields \textbf{89}, 347--386 (1991). 

\bibitem{BBP2} Bertini, L., Butt\`a, P., Pisante A.: \textit{Stochastic Allen-Cahn approximation of the mean curvature
    flow: large deviations upper bound.} Arch. Ration. Mech. Anal. \textbf{224}, 659--707 (2017).


\bibitem{BPRS} Bertini, L., Presutti, E., R\"udiger, B., Saada, E.: \textit{Dynamical fluctuations at the critical point: convergence to a nonlinear stochastic PDE.} Teor. Veroyatnost. i Primenen. \textbf{38} 689--741 (1993); translation in Theory Probab. Appl. \textbf{38} 586--629 (1993). 

\bibitem{Bi} Billingsley, P.: {\it Convergence of Probability Measures.} Wiley, New York 1968.

\bibitem{Cerrai} Cerrai, S.: \textit{Stochastic reaction-diffusion systems with multiplicative noise and non-Lipschitz reaction term.} Probab. Theory Related Fields \textbf{125}, 271--304 (2003).


\bibitem{CD} Cerrai, S., Debussche, A.: \textit{Large deviations for the dynamic $\Phi_d^{2n}$ model.} https://arxiv.org/pdf/1705.00541.pdf

\bibitem{DD} Da Prato, G., Debussche, A.: \textit{Strong solutions to the stochastic quantization equations.} Ann. Probab. \textbf{31}, 1900--1916 (2003).

\bibitem{DaZ} Da Prato, G., Zabczyk, J.: \textit{Stochastic equations in infinite dimensions.} Encyclopedia of Mathematics and its Applications, \textbf{44}. Cambridge University Press, Cambridge, 1992. 

\bibitem{DOPT} De Masi, A., Orlandi, E., Presutti, E., Triolo, L.:
 \textit{Motion by curvature by scaling nonlocal evolution equations.}  
J. Statist. Phys. \textbf{73}, 543--570 (1993). 

\bibitem{ESS} Evans, L. C., Soner, H. M. , Souganidis, P. E.:  \textit{Phase transitions and generalized motion by mean curvature.} Comm. Pure Appl. Math. \textbf{45}, 1097--1123 (1992).

\bibitem{FR} Fritz, J., R\"udiger, B.: \textit{Time dependent critical fluctuations of a one-dimensional local mean field model.} Probab. Theory Related Fields \textbf{103}, 381--407 (1995).

\bibitem{G} Gess, B.: \textit{Strong solutions for stochastic partial differential equations of gradient type.}  J. Funct. Anal. \textbf{263}, 2355--2383 (2012). 

\bibitem{Hairer} Hairer, M.: \textit{A theory of regularity structures.} Invent. Math. \textbf{198}, 269--504 (2014).

\bibitem{H} Hausenblas, E.: \textit{Approximation for semilinear stochastic evolution equations. } Potential Anal. \textbf{18}, 141--186 (2003). 

\bibitem{HH}  Hohenberg, P.C., Halperin, B.I. \textit{Theory of dynamic critical phenomena.} Rev. Mod. Phys. \textbf{49}, 435--479 (1977).

\bibitem{Ilmanen} Ilmanen, T.: \textit{Convergence of the Allen-Cahn equation to the BrakkeÕs motion by mean curvature.} J. Diff. Geom. \textbf{31}, 417--461 (1993).

\bibitem{J} Jentzen, A.: \textit{Pathwise Numerical Approximations of SPDEs with Additive Noise under Non-global Lipschitz Coefficients.} Potential Anal. \textbf{31}, 375--404 (2009).

\bibitem{JM} Jona-Lasinio, G., Mitter, P.K.: \textit{On the stochastic quantization of field theory.} Comm. Math. Phys. \textbf{101}, 409--436 (1985).

\bibitem{KS} Katsoulakis, M.A., Souganidis, P.E.:
\textit{Generalized motion by mean curvature as a macroscopic limit of stochastic Ising models with long range interactions and Glauber dynamics.} 
Comm. Math. Phys. \textbf{169}, 61--97 (1995). 

\bibitem{KORV} Kohn, R., Otto, F., Reznikoff, M.G., Vanden-Eijnden, E.: \textit{Action minimization and sharp-interface limits for the stochastic Allen-Cahn equation.} Comm. Pure Appl. Math. \textbf{60}, 393-438 (2007). 

\bibitem{KLL} Kov\'acs, M., Larsson, S., Lindgren, F.: \textit{On the backward Euler approximation of the stochastic Allen-Cahn equation.} J. Appl. Probab. \textbf{52}, 323--338 (2015). 

\bibitem{KR} Krylov, N.V., Rozovskii, B.L.: \textit{Stochastic Evolution Equations.} Plenum Publishing Corp., 1981; Translated from Itogi Nauki i Tekhniki, Seriya Sovremennye Problemy Matematiki \textbf{14}, 71--146 (1979).

\bibitem{L} Liu, W.: \textit{Well-posedness of stochastic partial differential equations with Lyapunov condition.} J. Differential Equations \textbf{255}, 572--592 (2013). 

\bibitem{LR} Liu, W., Ršckner, M.: \textit{SPDE in Hilbert space with locally monotone coefficients.} J. Funct. Anal. \textbf{259}, 2902--2922 (2010). 

\bibitem{Mariani} Mariani, M.: Large deviations principles for stochastic scalar conservation laws. Probab. Theory Related Fields \textbf{147}, 607--648 (2010). 
 
\bibitem{MW2} Mourrat J.-C., Weber, H.: \textit{Convergence of the two-dimensional dynamic Ising-Kac model to $\Phi^4_2$.} Preprint 2014, {\tt  arXiv:1410.1179}.

\bibitem{MW3} Mourrat J.-C., Weber, H.: \textit{The dynamic $\Phi^4_3$ model comes down from infinity.} Preprint 2017, {\tt  arxiv:1601.01234}.

\bibitem{Pazy} Pazy, A.: \textit{Semigroups of linear operators and applications to partial differential equations}. Applied Mathematical Sciences, \textbf{44}. Springer-Verlag, New York, 1983.

\bibitem {PrRo} Pr\'ev\^ot, C. , R\"ockner, M.: \textit{A concise course on stochastic partial differential equations.} Lecture Notes in Mathematics, \textbf{1905}. Springer, Berlin, 2007.

\bibitem{RY} D. Revuz, M. Yor: \textit{Continuous martingales and Brownian motion. Third edition}. Springer-Verlag, Berlin 1999.

\bibitem{RSZ} R\"ockner, M., Schmuland, B., Zhang, X: \textit{Yamada-Watanabe theorem for stochastic evolution equations in infinite dimensions.} Condensed Matter Physics \textbf{11}, 247--259 (2008).


\bibitem{Spohn} Spohn, H.: \textit{Large scale dynamics of interacting particles.} Springer, Berlin 1991.

\bibitem{S} Spohn, H.: \textit{Interface motion in models with stochastic dynamics.} J. Stat. Phys. \textbf{71}, 1081--1132 (1993).

\bibitem{T} Trotter, H. F.: \textit{Approximation of semi-groups of operators.} Pacific J. Math. \textbf{8},  887--919 (1958). 

\bibitem{YW} Yamada, T., Watanabe, S.: \textit{On the uniqueness of solutions of stochastic differential equations.} J. Math. Kyoto Univ. \textbf{11}, 155--167 (1971). 

\end{thebibliography}
\end{document}